\documentclass[11pt]{article}

\usepackage{amsmath,amssymb,amsthm,graphicx,color}

\DeclareMathOperator{\asc}{asc}
\DeclareMathOperator{\ltormax}{lrmax}
\DeclareMathOperator{\inv}{inv}

\DeclareMathOperator{\afterone}{afterone}
\DeclareMathOperator{\lastentry}{lastentry}
\DeclareMathOperator{\zerozeros}{zerozeros}
\DeclareMathOperator{\oneones}{oneones}
\DeclareMathOperator{\ones}{ones}
\DeclareMathOperator{\lbl}{label}

\newcommand{\qbinom}[2]{\genfrac{[}{]}{0pt}{}{#1}{#2}}

\def\fish{f}
\def\fishmap{g}

\newtheorem{theorem}{Theorem}[section]
\newtheorem{corollary}[theorem]{Corollary}
\newtheorem{lemma}[theorem]{Lemma}
\newtheorem{proposition}[theorem]{Proposition}
\newtheorem{definition}[theorem]{Definition}
\newtheorem{openproblem}[theorem]{Open Problem}
\newtheorem{conjecture}[theorem]{Conjecture}

\title{Pattern-Avoiding Fishburn Permutations and Ascent Sequences}

\author{Eric S. Egge\\
Department of Mathematics and Statistics\\
Carleton College\\
1 North College Street\\
Northfield, MN 55057  USA\\
\\
eegge@carleton.edu}

\begin{document}

\maketitle

\begin{abstract}
A Fishburn permutation is a permutation which avoids the bivincular pattern $(231, \{1\}, \{1\})$, while an ascent sequence is a sequence of nonnegative integers in which each entry is less than or equal to one more than the number of ascents to its left.
Fishburn permutations and ascent sequences are linked by a bijection $\fishmap$ of Bousquet-M\'elou, Claesson, Dukes, and Kitaev.
We write $F_n(\sigma_1,\ldots,\sigma_k)$ to denote the set of Fishburn permutations of length $n$ which avoid each of $\sigma_1,\ldots,\sigma_k$ and we write $A_n(\alpha_1,\ldots,\alpha_k)$ to denote the set of ascent sequences which avoid each of $\alpha_1,\ldots,\alpha_k$.
We settle a conjecture of Gil and Weiner by showing that $\fishmap$ restricts to a bijection between $F_n(3412)$ and $A_n(201)$.
Building on work of Gil and Weiner, we use elementary techniques to enumerate $F_n(123)$ with respect to inversion number and number of left-to-right maxima, obtaining expressions in terms of $q$-binomial coefficients, and to enumerate $F_n(123,\sigma)$ for all $\sigma$.
We use generating tree techniques to study the generating functions for $F_n(321, 1423)$, $F_n(321,3124)$, and $F_n(321,2143)$ with respect to inversion number and number of left-to-right maxima.
We use these results to show $|F_n(321,1423)| = |F_n(321,3124)| = F_{n+2} - n - 1$, where $F_n$ is a Fibonacci number, and $|F_n(321,2143)| = 2^{n-1}$.
We conclude with a variety of conjectures and open problems.
\end{abstract}

\section{Introduction}

A {\em Fishburn permutation} is a permutation $\pi$, written in one-line notation, with no subsequence $\pi_j$, $\pi_{j+1}$, $\pi_k$ for which $j + 1 < k$, $\pi_j = \pi_k + 1$, and $\pi_{j+1} > \pi_j$.
In the language of Bousquet-M\'elou, Claesson, Dukes, and Kitaev \cite{BMCDK}, a Fishburn permutation is a permutation which avoids the bivincular pattern $(231, \{1\}, \{1\})$.
We will write $\fish$ to denote this pattern, and we will refer to it as the {\em Fishburn pattern}.
In Figure \ref{fig:fishpattern} we have the graph of $\fish$;  here bold lines indicate required adjacencies.
\begin{figure}[ht]
\centering
\includegraphics[width=.8in]{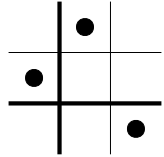}
\caption{The graph of the Fishburn pattern $\fish$.}
\label{fig:fishpattern}
\end{figure}
As an example, $3142$ is a Fishburn permutation, but $35142$ is not, since $3$, $5$, and $2$ form a copy of $\fish$.

Fishburn permutations are named for the fact that they are counted by the Fishburn numbers.
Bousquet-M\'elou, Claesson, Dukes, and Kitaev \cite{BMCDK} appear to be the first to have proved this, though they do not state the fact explicitly.
The Fishburn numbers may be defined by the fact \cite{Z} that their generating function is given by
\[ 1 + \sum_{m=1}^\infty \prod_{j=1}^m \left(1 - (1 - t)^j\right); \]
they are the terms of sequence A022493 in the Online Encyclopedia of Integer Sequences (OEIS).
In Table \ref{table:Fishburn} we have the first nine Fishburn numbers.
\begin{table}[ht]
\centering
\begin{tabular}{c|c|c|c|c|c|c|c|c|c}
$n$ & 0 & 1 & 2 & 3 & 4 & 5 & 6 & 7 & 8 \\
\hline
Fishburn number & 1 & 1 & 2 & 5 & 15 & 53 & 217 & 1014 & 5335
\end{tabular}
\caption{The first nine Fishburn numbers.}
\label{table:Fishburn}
\end{table}

In addition to Fishburn permutations, the Fishburn numbers are known to count a variety of other combinatorial objects, including ascent sequences \cite{BMCDK}, $({\bf 2} + {\bf 2})$-free posets \cite{BMCDK,DJK,DKRS}, linearized chord diagrams of degree $n$ \cite{Z}, nonisomorphic interval orders on $n$ unlabeled points \cite{Fishpaper,Fishbook}, and upper triangular matrices of any size with nonnegative integer entries and no zero rows or columns in which the sum of all entries is $n$ \cite{DP}.
Several authors have studied connections among these objects \cite{DM2,DM1,L}.
Like the set of Fishburn permutations, which is a superset of the set of permutations which avoid $231$, many of these objects are supersets of sets known to be enumerated by the Catalan numbers.

Fishburn numbers are also of number theoretic interest.
For example, Andrews and Sellers have shown \cite{AndrewsSellers} that they satisfy congruences similar to those which hold for the partition function $p(n)$, which gives the number of integer partitions of $n$, and Andrews and Jel\'{i}nek have studied \cite{AJ} some related $q$-series identities.

Among the combinatorial objects known to be in bijection with Fishburn permutations, we will be especially interested in ascent sequences.
In a sequence $a_1 \cdots a_n$ of integers, an {\em ascent} is an index $j$ such that $a_j < a_{j+1}$;  we write $\asc(a_1 \cdots a_n)$ to denote the number of ascents in $a_1 \cdots a_n$.
An {\em ascent sequence} is a sequence $a_1 \cdots a_n$ of nonnegative integers such that $a_1 = 0$ and for each $j \ge 2$, we have $a_j \le 1 + \asc(a_1 \cdots a_{j-1})$.
For example, $001021301$ is an ascent sequence (of length 9) but $001031201$ is not.

Bousquet-M\'elou, Claesson, Dukes, and Kitaev \cite{BMCDK} have given a natural bijection between Fishburn permutations and ascent sequences.
To describe it, suppose $\pi$ is a Fishburn permutation of length $n-1$.
A {\em site} in $\pi$ is a space between two entries in $\pi$, before the first entry of $\pi$, or after the last entry of $\pi$;  note that $\pi$ has $n$ sites,
We say a site is {\em active} whenever the permutation we obtain by inserting $n$ into that site is also a Fishburn permutation.
In a Fishburn permutation we number the active sites from left to right, starting with $0$.
For example, $415326$ is a Fishburn permutation whose five active sites are numbered as in Figure \ref{fig:activesites}.
\begin{figure}[ht]
\centering
$^04\ 1^15^23\ 2^36^4$
\caption{The numbering of the active sites in the Fishburn permutation $415326$.}
\label{fig:activesites}
\end{figure}

We can construct each Fishburn permutation $\pi$ of length $n$ uniquely by successively inserting $1,2,\ldots,n$ into active sites.
We write $\fishmap(\pi)$ to denote the sequence $a_1 \cdots a_n$ of labels of these active sites.
For example, $\fishmap(4175326) = 0110132$.
Bousquet-M\'elou, Claesson, Dukes, and Kitaev \cite{BMCDK} have shown that $\fishmap$ is a bijection between the set of Fishburn permutation of length $n$ and the set of ascent sequences of length $n$.

We will be interested in permutations and ascent sequences which avoid certain permutations or sequences.
In general, we say a sequence $\alpha$ of nonnegative integers {\em contains} a sequence $\beta$ whenever $\alpha$ has a subsequence with the same length and relative order as $\beta$.
We say $\alpha$ {\em avoids} $\beta$ whenever $\alpha$ does not contain $\beta$.
For example, the ascent sequence $0110132$ contains both $000$ (in the form of the subsequence $111$) and $021$ (in the form of the subsequence $032$ and three subsequences $132$) but it avoids $210$.
Similarly, the permutation $415326$ contains $231$ (in the form of the subsequences $453$ and $452$) and avoids $2413$.
For any sequences $\alpha_1, \ldots, \alpha_k$ of nonnegative integers we write $A_n(\alpha_1,\ldots, \alpha_k)$ to denote the set of ascent sequences of length $n$ which avoid each of $\alpha_1,\ldots, \alpha_k$.
Similarly, for any permutations $\sigma_1, \ldots, \sigma_k$ we write $S_n(\sigma_1,\ldots, \sigma_k)$ to denote the set of permutations of length $n$ which avoid each of $\sigma_1,\ldots, \sigma_k$, and we write $F_n(\sigma_1,\ldots, \sigma_k)$ to denote the set of Fishburn permutations of length $n$ which avoid each of $\sigma_1,\ldots, \sigma_k$.

In some situations we will be interested in relationships between $F_n(\sigma_1,\ldots,\sigma_k)$ and $A_n(\alpha_1,\ldots, \alpha_\ell)$.
The cases we consider are limited and our arguments are ad hoc, but recently Cerbai and Claesson have described \cite{CCnew} a general method of establishing similar relationships.
In particular, Cerbai and Claesson give an explicit construction which transforms a basis $\sigma_1,\ldots,\sigma_k$ for a set of pattern-avoiding Fishburn permutations to a basis for a corresponding set of pattern-avoiding modified ascent sequences.
Modified ascent sequences are in bijection with ascent sequences, and often this latter basis can be translated to a basis $\alpha_1,\ldots,\alpha_\ell$ for a set of pattern-avoiding ascent sequences.

In most situations we will be interested in enumerating $F_n(\sigma_1,\ldots,\sigma_k)$ or $A_n(\alpha_1,\ldots, \alpha_\ell)$ independent of their relationship.
In addition, in some cases we will refine these enumerations by studying generating functions for these sets with respect to specific statistics.
For permutations we will primarily be interested in two statistics:  the inversion number and the number of left-to-right maxima.
We recall that an {\em inversion} in a permutation $\pi$ is an ordered pair $(j,k)$ such that $j < k$ and $\pi_j > \pi_k$.
We write $\inv(\pi)$ to denote the number of inversions in $\pi$.
We also recall that a {\em left-to-right maximum} in $\pi$ is an entry of $\pi$ which is greater than all of the entries to its left.
We write $\ltormax(\pi)$ to denote the number of left-to-right maxima in $\pi$.
For example, $\inv(5176243) = 12$ and $\ltormax(3265174) = 3$.

Beginning with Duncan and Steingr\'{i}msson \cite{DS}, several additional authors have studied pattern-avoiding ascent sequences \cite{BP,CMS,CDDDS,FJLYZ,Lin,LF,LY,MS,MSFib,Pudwell}, statistics on ascent sequences \cite{CDDDS,FJLYZ,JS}, and connections between these sets and other combinatorial objects \cite{CMS,Yan}.
The study of pattern-avoiding Fishburn permutations was initiated by Gil and Weiner \cite{GW}, though they are also used as a bridge in work of Chen, Yan, and Zhou \cite{CYZ}.

We open our main results in Section \ref{sec:Fishisirrelevant} by giving an elementary argument that if $\pi$ is a Fishburn permutation which contains $231$ then in fact $\pi$ contains $3142$.
This allows us to use results of Simion and Schmidt \cite{SS} to recover several results of Gil and Weiner \cite{GW}.
In Section \ref{sec:3412and201} we prove a conjecture of Gil and Weiner \cite{GW} by showing that $\fishmap$ restricts to a bijection between $F_n(3412)$ and $A_n(201)$.
In Section \ref{sec:refineFn123} we obtain the generating function for $F_n(123)$ with respect to inversion number and number of left-to-right maxima.
In particular, we show 
\[ \sum_{\pi \in F_n(123)} q^{\inv(\pi)} t^{\ltormax(\pi)} r^{\afterone(\pi)} = t \sum_{s=0}^{n-2} q^{s^2+s-ns+\binom{n}{2}} \qbinom{n-2}{s}_q r^s + t^2 \sum_{s=1}^{n-1} q^{s^2-ns+\binom{n}{2}} \qbinom{n-2}{s-1}_q r^s. \]
Here $\afterone(\pi)$ is the number of entries of $\pi$ to the right of the $1$ and $\qbinom{n}{k}_q$ is the usual $q$-binomial coefficient.
In Section \ref{sec:A012beta} we study $A_n(012,\beta)$ for various binary sequences $\beta$, and in Section \ref{sec:Fn123sigma} we use these results to enumerate $F_n(123,\sigma)$ for any permutation $\sigma$.
The resulting sequences are sums of binomial coefficients which are connected with cake-cutting problems.
Notably, when $\sigma$ avoids $123$ but contains $\fish$, the enumeration of $F_n(123, \sigma)$ depends on whether $\sigma$ contains $2413$ or $3412$.

Beginning in Section \ref{sec:Fn3211432}, we use generating trees to obtain enumerations of sets of pattern-avoiding Fishburn permutations of the form $F_n(321,\sigma)$.
In Section \ref{sec:Fn3211432} we use generating trees to find the generating function for $F_n(321,1432)$ with respect to inversion number and number of left-to-right maxima.
It follows from this result that $|F_n(321,1432)| = F_{n+2} - n - 1$.
Here $F_n$ is the $n$th Fibonacci number, which we define by setting $F_0 = F_1 = 1$ and $F_n = F_{n-1} + F_{n-2}$ for $n \ge 2$.
In Section \ref{sec:Fn3213124} we use generating trees to find the generating function for $F_n(321,3124)$ with respect to inversion number and number of left-to-right maxima.
It follows from this result that $|F_n(321,3124)| = F_{n+2} - n - 1$.
This is the same formula we obtained for $|F_n(321,1432)|$, though the two generating trees are different.
In particular, the generating tree for $F_n(321,1432)$ has finitely many labels, while the generating tree for $F_n(321,3124)$ has infinitely many.
In Section \ref{sec:Fn3212143} we use generating trees to show the generating function for $F_n(321,2143)$ with respect to the number of left-to-right maxima is given by
\[ \sum_{\pi \in F_n(321,2143)} t^{\ltormax(\pi)} =  t (t + 1)^{n-1}, \]
from which it follows that $|F_n(321,2143)| = 2^{n-1}$.

We conclude in Section \ref{sec:opc} with a selection of open problems and conjectures.

\section{When Fishburn Means $231$-Avoiding}
\label{sec:Fishisirrelevant}

Gil and Weiner \cite{GW} have enumerated $F_n(\sigma_1,\ldots,\sigma_k)$ for a variety of $\sigma_1,\ldots, \sigma_k$ of lengths three and four.
In comparing their results with those of Simion and Schmidt \cite{SS}, one notes that in some cases we have $F_n(\sigma_1,\ldots, \sigma_k) = S_n(231,\sigma_1,\ldots,\sigma_k)$.
In this section we characterize those situations, by showing that this occurs if and only if at least one of $\sigma_1,\ldots, \sigma_k$ is contained in $3142$.
This allows us to recover several of Gil and Weiner's enumerations.
We begin by studying Fishburn permutations which contain 231.

\begin{lemma}
\label{lem:3142}
If $\pi$ is a permutation which contains $231$ but avoids $\fish$ then $\pi$ contains $3142$.
\end{lemma}
\begin{proof}
Suppose by way of contradiction that $\pi$ contains $231$ but avoids both $\fish$ and $3142$.

Among all copies of $231$ in $\pi$, consider those in which the largest element is as large as possible.
Among these copies of $231$, consider those in which the leftmost entry is as far to the right as possible.
And among these copies of $231$, choose the one in which the rightmost entry is as large as possible;  let the entries in this copy be $a$, $b$, and $c$, from left to right.

In Figure \ref{fig:first3142diagram} we have the diagram of $\pi$ divided into regions according to the positions of $a$, $b$, and $c$.
We have numbered the regions where other entries of the permutation can be for ease of reference.
\begin{figure}
\begin{center}
\includegraphics[width=128pt]{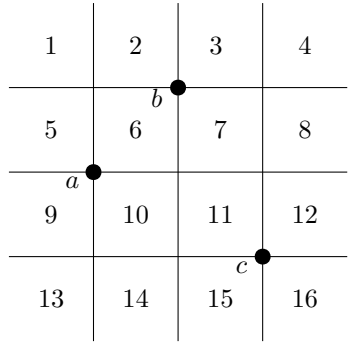}
\end{center}
\caption{The permutation $\pi$ in the proof of Lemma \ref{lem:3142} with the chosen copy $abc$ of $231$.}
\label{fig:first3142diagram}
\end{figure}

By our choice of $b$, boxes 2 and 3 must be empty.
Similarly, by our choice of $a$, boxes 6 and 10 must be empty.
And by our choice of $c$, boxes 11 and 12 must be empty.
This leaves us with the diagram in Figure \ref{fig:second3142diagram}, in which the shaded regions must be empty.
\begin{figure}
\begin{center}
\includegraphics[width=128pt]{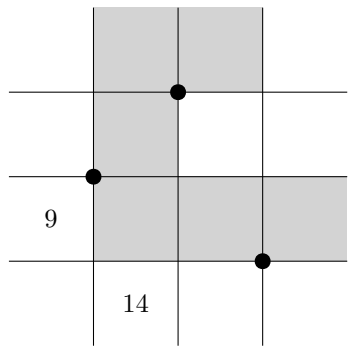}
\end{center}
\caption{The permutation $\pi$ in the proof of Lemma \ref{lem:3142} with some empty regions shaded.}
\label{fig:second3142diagram}
\end{figure}

Now if box 14 contains an entry $x$, then $a, x, b, c$ is a copy of $3142$, which contradicts our assumption that $\pi$ avoids $3142$.
So box 14 must be empty and $a$ and $b$ must be adjacent.

On the other hand, if $a \neq c+1$ then $c+1$ must be in box 9.
In this case $c+1$ is the smallest entry in its box.
Hence $c+1$, the entry immediately to its right, and $c$, form a copy of $\fish$.
This contradicts our assumption that $\pi$ does not contain $\fish$.

Finally, if $a = c+1$ then $a,b,c$ is a copy of $\fish$, which is also a contradiction.
\end{proof}

We note that Lemma \ref{lem:3142} is implicit in Gil and Weiner's proof of \cite[Theorem 3.4]{GW}.

Lemma \ref{lem:3142} allows us to characterize those sets of forbidden patterns for which avoiding $f$ is equivalent to avoiding $231$.

\begin{proposition}
\label{prop:fis231if3142}
For any permutations $\sigma_1,\ldots,\sigma_k$ the following are equivalent.
\begin{enumerate}
\item[{\upshape (i)}]
$F_n(\sigma_1,\ldots, \sigma_k) = S_n(231, \sigma_1,\ldots,\sigma_k)$ for all $n \ge 0$.
\item[{\upshape (ii)}]
At least one of $\sigma_1,\ldots, \sigma_k$ is contained in $3142$.
\end{enumerate}
\end{proposition}
\begin{proof}
(i) $\Rightarrow$ (ii)
If $F_n(\sigma_1,\ldots, \sigma_k) = S_n(231, \sigma_1,\ldots,\sigma_k)$ for all $n \ge 0$ then in particular $3142 \not\in F_4 (\sigma_1,\ldots, \sigma_k)$, since $3142$ contains $231$ and so is not in $S_n(231,\sigma_1,\ldots,\sigma_k)$.
But $3142$ does not contain $\fish$, so it must contain at least one of $\sigma_1,\ldots,\sigma_k$.

(ii) $\Rightarrow$ (i)
Since $\fish$ is a special type of $231$, we must have $S_n(231,\sigma_1,\ldots, \sigma_k) \subseteq F_n(\sigma_1,\ldots, \sigma_k)$ for all $n \ge 0$.
To prove the reverse inclusion, suppose $\pi \in F_n(\sigma_1,\ldots,\sigma_k)$.
If $\pi \not\in S_n(231,\sigma_1,\ldots, \sigma_k)$ then $\pi$ contains $231$.
By Lemma \ref{lem:3142} this means $\pi$ contains $3142$.
Therefore, $\pi$ does not avoid each of $\sigma_1,\ldots,\sigma_k$, which contradicts the fact that $\pi \in F_n(\sigma_1,\ldots, \sigma_k)$.
It follows that $\pi \in S_n(231,\sigma_1,\ldots,\sigma_k)$, which is what we wanted to prove.
\end{proof}

Proposition \ref{prop:fis231if3142} enables us to use results of Simion and Schmidt \cite{SS} to recover several results of Gil and Weiner \cite{GW}.

\begin{corollary}
\cite[Theorem 2.1]{GW}
\label{cor:F132or213or312}
For each $\sigma \in \{132, 213, 312\}$ we have $|F_n(\sigma)| = 2^{n-1}$ for all $n \ge 1$.
\end{corollary}
\begin{proof}
By Proposition \ref{prop:fis231if3142} we have $F_n(\sigma) = S_n(231,\sigma)$ for each $\sigma \in \{132, 213, 312\}$.
Now the case $\sigma = 132$ follows from Proposition 9 in \cite{SS}, the case $\sigma = 213$ follows from Lemma 5(d) and Proposition 12 in \cite{SS}, and the case $\sigma = 312$ follows from Lemma 5(b) and Proposition 12 in \cite{SS}.
\end{proof}

\section{The Correspondence between $F_n(3412)$ and $A_n(201)$}
\label{sec:3412and201}

In the last section of \cite{GW}, Gil and Weiner suggest investigating the relationship between pattern-avoid Fishburn permutations and pattern-avoiding ascent sequences.
As a step in this direction, they seem to suggest $A_n(120)$ might be in bijection with $F_n(4123)$ for all $n \ge 0$.
However, a computer search shows that $|A_{10}(120)| = 20754$ while $|F_{10}(4123)| = 20753$.
On the other hand, they also allude to data in their Table 3 showing that $|F_n(3412)| = |A_n(201)|$ for $0 \le n \le 8$.
In this section we show that the map $\fishmap$ of Bousquet-M\'elou, Claesson, Dukes, and Kitaev from Fishburn permutations to ascent sequences restricts to a bijection from $F_n(3412)$ to $A_n(201)$.
We note that Cerbai and Claesson \cite{CCnew} have recently used a more general approach to address this problem.

We begin by describing how inserting an entry into an active site in a Fishburn permutation affects the set of active sites.

\begin{lemma}
\label{lem:activetrackinfish}
Suppose $n \ge 1$, $\sigma \in S_{n-1}$ is a Fishburn permutation, and $\pi$ is the Fishburn permutation we obtain by inserting $n$ into an active site of $\sigma$.
Then the following hold.
\begin{enumerate}
\item[{\upshape (i)}]
If a site is not active in $\sigma$ then it is not active in $\pi$.
\item[{\upshape (ii)}]
If a site is active in $\sigma$ and it is not the site into which $n$ was inserted to obtain $\pi$, then it is active in $\pi$.
\item[{\upshape (iii)}]
The site immediately to the left of $n$ in $\pi$ is active.
\item[{\upshape (iv)}]
The site immediately to the right of $n$ in $\pi$ is active if and only if $n$ is to the right of $n-1$ in $\pi$.
\end{enumerate}
\end{lemma}
\begin{proof}
(i)
If a site is not active in $\sigma$ then inserting $n$ into that site creates a copy $a, n, a-1$ of $\fish$.
If we insert $n + 1$ into that site then it will be adjacent to $a$ (since we inserted $n$ elsewhere to obtain $\pi$) and $a, n+1, a-1$ will also be a copy of $\fish$. 
Therefore the site is not active in $\pi$.

(ii)
We prove the contrapositive.
If a site is not active in $\pi$ then inserting $n+1$ into it will create a copy $a, n+1, a-1$ of $\fish$.
Since the site is not adjacent to $n$ in $\pi$, we must have $a < n$.
So if we insert $n$ into the site in $\sigma$ then $a, n, a-1$ will be a copy of $\fish$, which means the site is not active in $\sigma$.

(iii)
Suppose by way of contradiction that the site immediately to the left of $n$ in $\pi$ is not active.
Then inserting $n+1$ into the site will create a copy $a, n+1, a-1$ of $\fish$.
But $a$ is to the left of $n+1$, which is to the left of $n$, so $a, n, a-1$ is a copy of $\fish$ in $\pi$, contradicting our assumption that $\pi$ avoids $\fish$.

(iv)
The permutation we obtain from $\pi$ by inserting $n+1$ immediately to the right of $n$ can contain $\fish$ if and only if $n, n+1, n-1$ is a copy of $\fish$.
But this occurs if and only if $n$ is to the left of $n-1$ in $\pi$.
\end{proof}

Lemma \ref{lem:activetrackinfish} allows us to connect the relative magnitudes of some entries of $\fishmap(\pi)$ with the order in which the corresponding entries appear in $\pi$.

\begin{lemma}
\label{lem:Fishascentordering}
Suppose $\pi \in S_n$ is a Fishburn permutation with $\fishmap(\pi) = \alpha$.
If $j < k$ and $\alpha_k \le \alpha_j$ then $k$ is to the left of $j$ in $\pi$.
\end{lemma}
\begin{proof}
Let $\sigma \in S_j$ be the Fishburn permutation we obtain from $\pi$ by removing all entries greater than $j$.
Note that $\fishmap(\sigma) = \alpha_1 \cdots \alpha_j$.
By Lemma \ref{lem:activetrackinfish}(iii) we know the site immediately to the left of $j$ in $\sigma$ is active, and by Lemma \ref{lem:activetrackinfish}(i)(ii) we know its label is $\alpha_j$.
Furthermore, each time we insert a new entry there will be at least $\alpha_j + 1$ active sites to the left of $j$, by Lemma \ref{lem:activetrackinfish}(i)--(iii).
In other words, at every stage the sites with smaller labels (as well as, potentially, some with larger labels) are all to the left of $j$.
In particular, if $\alpha_k \le \alpha_j$ then $k$ is to the left of $j$.
\end{proof}

Lemma \ref{lem:activetrackinfish} also allows us to show that the label on a given active site can only grow as we insert entries in other sites.

\begin{lemma}
\label{lem:watchasite}
Suppose $\sigma \in S_k$ is a Fishburn permutation, $\pi \in S_n$ is a Fishburn permutation we obtain from $\sigma$ by inserting $k+1,\ldots,n$, and $j \le k+1$.
If the $j$th site from the left is active in both $\sigma$ and $\pi$, the label of the site in $\sigma$ is $\beta_1$, and the label on the site in $\pi$ is $\beta_2$, then $\beta_1 \le \beta_2$.
\end{lemma}
\begin{proof}
Consider how the label on the $j$th site can change when we insert a new entry.
When we insert an entry into a site to the right, the label is unchanged by Lemma \ref{lem:activetrackinfish}(i)(ii).
When we insert an entry into the $j$th site the label is also unchanged, by Lemma \ref{lem:activetrackinfish}(i)--(iii).
And when we insert an entry into a site to the left, the label is either unchanged or increases by one, by Lemma \ref{lem:activetrackinfish}(i)--(iv).
\end{proof}

We are now ready to prove our main result of this section, which proves a conjecture of Gil and Weiner \cite{GW}.

\begin{proposition}
Suppose $\pi$ is a Fishburn permutation with $\fishmap(\pi) = \alpha$.
Then $\pi$ avoids $3412$ if and only if $\alpha$ avoids $201$.
In particular, $|F_n(3412)| = |A_n(201)|$ for all $n \ge 0$.
\end{proposition}
\begin{proof}
We prove the contrapositive.

Suppose $\pi \in S_n$ is a Fishburn permutation which contains $3412$;  we show $\fishmap(\pi)$ contains $201$.
Without loss of generality, assume that if we remove $n$ from $\pi$ then the resulting permutation avoids $3412$.
Then $n$ must play the role of the 4 in every copy of $3412$ in $\pi$.
Let $\pi_i$ be the largest entry that plays the role of the 2 in any copy of $3412$.
And let $\pi_j$ be any entry of $\pi$ which plays the role of the 3 in a copy of $3412$ in which $\pi_i$ plays the role of the 2.
Let $\alpha_i$ be the entry of $\alpha$ corresponding to the insertion of $\pi_i$ in $\pi$, let $\alpha_j$ be the entry corresponding to the insertion of $\pi_j$, and let $\alpha_n$ be the entry corresponding to the insertion of $n$.
Since $\pi_i < \pi_j < n$, from left to right in $\alpha$ we have $\alpha_i$, $\alpha_j$, and $\alpha_n$.
We claim $\alpha_i, \alpha_j, \alpha_n$ is a copy of $201$.
To prove this, we show $\alpha_j < \alpha_n$ and $\alpha_n < \alpha_i$.

The fact that $\alpha_j < \alpha_n$ follows from Lemma \ref{lem:Fishascentordering}, since $n$ was inserted to the right of $\pi_j$.

To show $\alpha_n < \alpha_i$, first note that when we insert $\pi_i$ there must be an entry to its left, since there must be an entry to play to role of 1 in our copy of $3412$.
Let one such entry be $\pi_k$.
When we insert $n$, it must be the first time we have inserted two entries greater than $\pi_i$ in increasing order to the left of $\pi_k$:  if it were not, then the two earlier entries, along with $\pi_k$ and $\pi_i$, would be a copy of $3412$, contradicting our choice of $n$.
Therefore, each time we insert a number $\ell$, $\pi_i < \ell < n$, to the left of $\pi_i$, we know $\ell -1$ is somewhere to its right.
Therefore, by Lemma \ref{lem:activetrackinfish}, none of these insertions increases  the number of active sites to the left of $\pi_i$.
And since $n$ is inserted to the left of $\pi_k$, which is to the left of $\pi_i$, we must have $\alpha_n < \alpha_i$.

Now suppose $\pi \in S_n$ is a Fishburn permutation and $f(\pi) = \alpha$ contains $201$;  we show $\pi$ contains $3412$.
Without loss of generality, we can assume $\alpha_n$ plays the role of the 1 in every copy of $201$ in $\alpha$.
Let $\alpha_k$ be the rightmost entry which plays the role of the 0 in any copy of $201$.
In addition, let $\alpha_j$ be any entry which plays the role of the 2 in any copy of $201$ that includes $\alpha_k$.
When $j$ is inserted in the construction of $\pi$, there will be an entry to its left, since $\alpha_j \ge 2$.
Let $i$ be the rightmost such entry.
We claim $i < j < k < n$ appear in the order $k,n,i,j$, and thus form a copy of $3412$.

By our choice of $\alpha_k$, every entry of $\alpha$ between $\alpha_k$ and $\alpha_n$ is greater than $\alpha_n$.
By Lemma \ref{lem:activetrackinfish} and the fact that $\alpha_n > \alpha_k$, we know $k$ appears to the left of $n$ in $\pi$.

By our choice of $i$ and Lemma \ref{lem:watchasite}, if $n$ were inserted to the right of $i$ then we would have $\alpha_n \ge \alpha_j$, which contradicts the fact that $\alpha_j$ plays the role of the 2 in our copy of 201 and $\alpha_n$ plays the role of the $1$.
Therefore $n$ appears to the left of $i$ in $\pi$.

By construction $i$ appears to the left of $j$ in $\pi$, which completes the proof.
\end{proof}

Using similar techniques, one can also show that $\fishmap$ is a bijection between $S_n(231)$ and $A_n(101)$, which is consistent with Duncan and Steingr\'{i}msson's result \cite[Corollary 2.12]{DS} that $|A_n(101)| = C_n$ for all $n \ge 0$.
Here $C_n = \frac{1}{n+1}\binom{2n}{n}$ is the $n$th Catalan number.
Duncan and Steingr\'{i}msson also show \cite[Theorem 2.11]{DS} that $|A_n(021)| = C_n$ for all $n \ge 0$.
This suggests the following open problem.

\begin{openproblem}
\label{op:021}
For $n \ge 0$, characterize the permutations $\pi \in S_n$ with $\fishmap(\pi) \in A_n(021)$.
\end{openproblem}

We note that the set of permutations in Open Problem \ref{op:021} is not characterized by pattern avoidance.
If it were, then it would have to be the set of permutations which avoid $231$, since there is no ascent sequence corresponding with $231$.
But $\fishmap(3142) = 0101$, and $3142$ does not avoid $231$.

For other results on $A_n(021)$, see Chen, Dai, Dokos, Dwyer, and Sagan \cite{CDDDS}.

\section{Refinements of $|F_n(123)|$}
\label{sec:refineFn123}

Suppose $\sigma_1,\ldots,\sigma_k$ are permutations.
One of our goals in this paper is to study the generating function for $F_n(\sigma_1,\ldots,\sigma_k)$ with respect to certain permutation statistics, especially inversion number and number of left-to-right maxima.
In Proposition \ref{prop:fis231if3142} we showed that for $F_n(132)$, $F_n(231)$, $F_n(213)$, and $F_n(312)$ this reduces to studying the corresponding generating function for $S_n(132, 231)$, $S_n(231)$, $S_n(231,213)$, and $S_n(231, 312)$, respectively.
But the situation is different for $F_n(123)$ and $F_n(321)$.
In this section we obtain the generating function for $F_n(123)$ with respect to inversion number and number of left-to-right maxima in terms of $q$-binomial coefficients.
We begin by describing $\fishmap(\pi)$ for $\pi \in F_n(123)$.

\begin{lemma}
\label{lem:Fn123ascentchar}
The following are equivalent for any Fishburn permutation $\pi$.
\begin{enumerate}
\item[{\upshape (i)}]
$\pi$ avoids $123$.
\item[{\upshape (ii)}]
All entries of $\fishmap(\pi)$ are $0$ or $1$.
\item[{\upshape (iii)}]
$\fishmap(\pi)$ avoids $012$.
\end{enumerate}
\end{lemma}
\begin{proof}
In each part we prove the contrapositive.

(i) $\Rightarrow$ (ii)
Suppose $\fishmap(\pi)$ contains an entry greater than one.
Note that in any Fishburn permutation, the leftmost site (to the left of all entries) is active, but the next site from the left that could be active is in the leftmost ascent.
Therefore, the entry inserted in $\pi$ for the entry of $\fishmap(\pi)$ which is greater than one must have been inserted to the right of the leftmost ascent.
This creates a copy of $123$, so $\pi$ contains $123$.

(ii) $\Rightarrow$ (iii)
If $\fishmap(\pi)$ contains $012$ then the three entries in any such subsequence are distinct.
Therefore, $\fishmap(\pi)$ cannot consist of only $0$s and $1$s.

(iii) $\Rightarrow$ (i)
Suppose $i < j < k$ form a copy of $123$ in $\pi$.
Since $j > i$ and $j$ is to the right of $i$, by Lemma \ref{lem:Fishascentordering} we have $\alpha_i < \alpha_j$.
Similarly, $\alpha_j < \alpha_k$.
Therefore, $\alpha_i, \alpha_j, \alpha_k$ is a copy of $012$ in $\fishmap(\pi)$. 
\end{proof}

Lemma \ref{lem:Fn123ascentchar} gives us a way to associate each subset of $\{2,3,\ldots,n\}$ with a permutation $\pi \in F_n(123)$.
In particular, for any set $A \subseteq \{2,3,\ldots,n\}$, let $\pi(A)$ be the permutation consisting of the elements of $A$ in decreasing order, followed by $1$, followed by the elements of $[n]-(A \cup \{1\})$ in decreasing order.
As we show next, this map is a bijection.

\begin{lemma}
\label{lem:piA123bijection}
For any $n \ge 1$, the map $\pi$ from the power set of $\{2,3,\ldots,n\}$ to $S_n$ is a bijection between the power set of $\{2,\ldots,n\}$ and $F_n(123)$.
\end{lemma}
\begin{proof}
Since $A$ is the set of entries to the left of $1$ in $\pi(A)$, we can recover $A$ from $\pi(A)$.
Therefore, it's sufficient to prove the following.
\begin{enumerate}
\item[(1)]
For all $A \subseteq \{2,\ldots,n\}$ we have $\pi(A) \in F_n(123)$.
\item[(2)]
If $\pi \in F_n(123)$ then there exists $A \subseteq \{2,\ldots,n\}$ with $\pi = \pi(A)$.
\end{enumerate}

\medskip
\noindent
Proof of (1):
By construction $\pi(A)$ consists of at most two decreasing subsequences.
Any copy of $123$ can have at most one entry in each decreasing subsequence, so $\pi(A)$ avoids $123$.
Similarly, if $abc$ is a copy of $\fish$ in $\pi(A)$ then $a$ must be in the first decreasing subsequence and $b$ must be in the second.
But $a$ and $b$ must also be adjacent, so $a = 1$.
Now $c \neq a-1$, so $\pi(A)$ avoids $\fish$.
\hfill $\diamond$

\medskip
\noindent
Proof of (2):
Suppose $n \ge 1$ and $\pi \in F_n(123)$.

If $1$ is the rightmost entry of $\pi$; we claim $\pi = n \cdots 2 1$.
If not, then let $a$ be the rightmost entry of $\pi$ with $a \neq \pi_{n+1-a}$.
Then the last $a-1$ entries of $\pi$ are $a-1, a-2, \ldots, 2, 1$.
If $b$ is the entry immediately to the right of $a$, then $b > a$, since all entries smaller than $a$ appear at the right end of $\pi$.
Furthermore, $a-1$ is among these entries, so $a, b, a-1$ is a copy of $\fish$, which contradicts the fact that $\pi$ avoids $\fish$.

If $1$ is not the rightmost entry of $\pi$ then the entries to the right of $1$ must be in decreasing order, since $\pi$ avoids $123$.
Now suppose there is an ascent $a, b$ to the left of $1$.
If $a-1$ is to the left of $a$ then $a-1, a, b$ is a copy of $123$, so $a-1$ must be to the right of $b$.
Now $a, b, a-1$ is a copy of $\fish$.
This contradicts the fact that $\pi$ avoids $\fish$, so the entries to the left of $1$ are in decreasing order.
Let $A$ be the set of these entries;  now we have $A \subseteq \{2,\ldots,n\}$ and $\pi = \pi(A)$.
\end{proof}

Next we show that $\pi$ and $\fishmap$ are related in a natural way:  when we apply $\fishmap$ to $\pi(A)$ we essentially obtain the complement of the indicator sequence for $A$.

\begin{proposition}
\label{prop:buildascentfromA}
Suppose $n \ge 1$ and  $A \subseteq \{2,\ldots,n\}$.
For $1 \le j \le n$, the $j$th entry of $\fishmap(\pi(A))$ is given by
\[ (g(\pi(A)))_j = \begin{cases}
0 & \text{if } j \in A \cup \{1\}; \\
1 & \text{if } j \not\in A \cup \{1\}.
\end{cases} \]
\end{proposition}
\begin{proof}
We can check the result for $n = 1$ and $n = 2$, so suppose by induction it holds for some $n \ge 3$ and fix $A \subseteq \{2,\ldots, n, n+1\}$.
Now set $B = A -\{n+1\}$.

Since $\pi(B)$ is $\pi(A)$ with $n+1$ removed, and $\fishmap(\pi(B))$ and $\fishmap(\pi(A))$ agree in their first $n$ entries, by induction the result holds for $1 \le j \le n$.
Next we note that $\pi(B)$ has two active sites:  one is at the left end and the other is immediately to the right of $1$.
The entry $n+1$ in $\pi(A)$ is in the first of these if and only if $n+1 \in A$, so the result also holds for $j = n+1$.
\end{proof}

Proposition \ref{prop:buildascentfromA} allows us to refine Gil and Weiner's enumeration of $F_n(123)$.
Our refinement will involve inversion numbers, numbers of left-to-right maxima, and the number of entries in $\pi$ to the right of $1$, which we denote by $\afterone(\pi)$.
We will connect these statistics to the following statistics on binary sequences.

\begin{definition}
Suppose $\alpha$ is a binary sequence.
\begin{enumerate}
\item
An {\em inversion} in $\alpha$ is an ordered pair $(j,k)$ such that $j <k$, $\alpha(j) = 1$, and $\alpha(k) = 0$.
We write $\inv(\alpha)$ to denote the number of inversions in $\alpha$.
\item
We write $\ones(\alpha)$ to denote the number of ones in $\alpha$.
\item
We write $\zerozeros(\alpha)$ to denote the number of ordered pairs $(j,k)$ with $j < k$ and $\alpha(j) = \alpha(k) = 0$.
\item
We write $\oneones(\alpha)$ to denote the number of ordered pairs $(j,k)$ with $j < k$ and $\alpha(j) = \alpha(k) = 1$.
\item
We write $\lastentry(\alpha)$ to denote the rightmost entry of $\alpha$.
\end{enumerate}
\end{definition}

As we might hope, some of our statistics on permutations correspond with statistics on binary sequences via $\fishmap$.

\begin{proposition}
\label{prop:statscorrespondence}
Suppose $n \ge 2$.
For any $\pi \in F_n(123)$, the following hold.
\begin{enumerate}
\item[{\upshape (i)}]
$\afterone(\pi) = \ones(\fishmap(\pi))$.
\item[{\upshape (ii)}]
$\ltormax(\pi) = 1 + \lastentry(\fishmap(\pi)) = \begin{cases} 1 & \text{if\ $\pi_1 = n$;} \\ 2 & \text{if\ $\pi_1 \neq n$.}\end{cases}$
\item[{\upshape (iii)}]
$\inv(\pi) = \inv(\fishmap(\pi)) + \zerozeros(\fishmap(\pi)) + \oneones(\fishmap(\pi))$.
\end{enumerate}
\end{proposition}
\begin{proof}
Let $A$ be the subset of $\{2,\ldots,n\}$ with $\pi = \pi(A)$.

(i)
Note that $\afterone(\pi)$ and $\ones(\fishmap(\pi))$ are both equal to $n-1-|A|$.

(ii)
We have
\[ \ltormax(\pi) = \begin{cases} 1 & \text{if\ $n \in A$}; \\ 2 & \text{if\ $n \not\in A$}; \end{cases} \]
and
\[ \lastentry(\fishmap(\pi)) = \begin{cases} 0 & \text{if\ $\pi_1 = n$}; \\ 1 & \text{if\ $\pi_1 \neq n$}. \end{cases} \]
Since $\pi_1 = n$ if and only if $n \in A$, the result follows.

(iii)
Note that $\inv(\pi)$ and $\inv(\fishmap(\pi)) + \zerozeros(\fishmap(\pi)) + \oneones(\fishmap(\pi))$ are both equal to the number of ordered pairs $(j,k)$ such that $j < k$ and one of the following holds:  $j,k \in A \cup \{1\}$, $j,k \not\in A \cup \{1\}$, or $j \not\in A \cup \{1\}$ and $k \in A \cup \{1\}$.
\end{proof}

We can use Proposition \ref{prop:statscorrespondence} to enumerate $F_n(123)$ with respect to number of entries to the right of $1$, left-to-right-maxima, and number of inversions.
We give these results in terms of $q$-binomial coefficients, so we first briefly review these polynomials.

\begin{definition}
For any nonnegative integers $m$ and $n$, the $q$-binomial coefficient is the generating function
\[ \qbinom{m+n}{m}_q = \sum_{\alpha \in B_{m+n,m}} q^{\inv(\alpha)}. \]
Here $B_{n,m}$ is the set of binary sequences of length $n \ge 0$ with exactly $m$ ones.
\end{definition}

The reader can check that a generalization of Pascal's identity holds:
\[ \qbinom{n}{m}_q = q^m \qbinom{n-1}{m}_q + \qbinom{n-1}{m-1}_q. \]
The $q$-binomial coefficients also satisfy a generalization of the binomial theorem:
\[ \prod_{k=0}^{n-1} (1+q^k x) = \sum_{k = 0}^n q^{\binom{k}{2}} \qbinom{n}{k}_q x^k. \]
We can prove this second result by showing that both sides are equal to 
\[ \sum_{\alpha \in B_n} q^{\inv(\alpha) + \oneones(\alpha)} x^{\ones(\alpha)},\]
where $B_n$ is the set of binary sequences of length $n$.
We obtain the product on the left by prepending entries to build $\alpha$:  prepending a $0$ to a binary sequence of length $k$ contributes a factor of $1$ and prepending a $1$ contributes a factor of $q^k x$.
We obtain the sum on the right by considering the generating function for the binary sequences of length $n$ with exactly $k$ ones.
The $q$-binomial coefficient records the inversions in these sequences while the $q^{\binom{k}{2}}$ records $\oneones(\alpha)$.

We are now ready to give the generating function for $F_n(123)$ with respect to inversion number, number of left-to-right maxima, and $\afterone$.

\begin{proposition}
For all $n \ge 2$ we have
\begin{equation}
\label{eqn:F123nqrt}
\sum_{\pi \in F_n(123)} q^{\inv(\pi)} t^{\ltormax(\pi)} r^{\afterone(\pi)} = t \sum_{s=0}^{n-2} q^{s^2+s-ns+\binom{n}{2}} \qbinom{n-2}{s}_q r^s + t^2 \sum_{s=1}^{n-1} q^{s^2-ns+\binom{n}{2}} \qbinom{n-2}{s-1}_q r^s. 
\end{equation}
In particular, we have
\begin{equation}
\label{eqn:F123nqr}
\sum_{\pi \in F_n(123)} q^{\inv(\pi)} r^{\afterone(\pi)} = \sum_{s=0}^{n-1} q^{s^2-ns+ \binom{n}{2}} \qbinom{n-1}{s}_q r^s. 
\end{equation}
\end{proposition}
\begin{proof}
By Proposition \ref{prop:statscorrespondence} we have
\[ \sum_{\pi \in F_n(123)} q^{\inv(\pi)} t^{\ltormax(\pi)} r^{\afterone(\pi)} = t \sum_{\alpha \in B_{n-2}} q^{\inv(0\alpha 0) + \zerozeros(0\alpha 0) + \oneones(0\alpha 0)} r^{\ones(0\alpha 0)}. \]
Note that if $\alpha \in B_{n-2}$ then $0 \le \ones(0\alpha 0) \le n-2$ and fix $s$ with $0 \le s \le n-2$.
Then the coefficient of $r^s t$ is the product of the following factors.
\begin{center}
\begin{tabular}{ccl}
$\qbinom{n-2}{s}_q$ & \hspace{5pt} & for inversions in $\alpha$ \\[3ex]
$q^{\binom{n-s}{2}}$ & & for copies of $00$ in $0\alpha 0$ \\[3ex]
$q^{\binom{s}{2}}$ & & for copies of $11$ in $0\alpha 0$ \\[3ex]
$q^s$ & & for inversions using the rightmost $0$ in $0\alpha 0$
\end{tabular}
\end{center}
Therefore, the coefficient of $r^s t$ is 
\[ q^{\binom{n-s}{2} + \binom{s}{2} + s} \qbinom{n-2}{s}_q = q^{s^2 + s -ns +\binom{n}{2}} \qbinom{n-2}{s}_q. \]
We obtain the coefficient of $r^s t^2$ in \eqref{eqn:F123nqrt} and the coefficient of $r^s$ in \eqref{eqn:F123nqr} in a similar fashion.
\end{proof}

At the beginning of this section we noted that obtaining generating functions for $F_n(132)$, $F_n(213)$, $F_n(231)$, and $F_n(312)$ is equivalent to obtaining the same generating functions for certain sets of classical pattern-avoiding permutations, but that this is not the case for $F_n(123)$ or $F_n(321)$.
We have obtained the generating function for $F_n(123)$ with respect to certain natural statistics, but it remains an open problem to obtain analogous generating functions for $F_n(321)$.

\section{Enumerating $A_n(012, \beta)$}
\label{sec:A012beta}

Proposition \ref{prop:fis231if3142} reduces the enumeration of Fishburn permutations avoiding at least one of $132$, $213$, $231$, or $312$ and any set of additional patterns to the enumeration of a closely related set of classical pattern-avoiding permutations.
Obtaining enumerations of these sets is well understood;  one approach is to use work of Mansour and Vainshtein \cite{MV}.
But again the situation is different for $123$ and $321$.
Our next goal is to enumerate $F_n(123,\sigma)$ for all $\sigma$, but to accomplish this we first need to enumerate $A_n(012,\beta)$ for various sequences $\beta$.
With Lemma \ref{lem:Fn123ascentchar} in mind, we start by considering a different type of containment relation on binary sequences.

\begin{definition}
\label{defn:bsavoid}
We say a binary sequence $\alpha$ {\em contains} a binary sequence $\beta$ whenever $\alpha$ contains a subsequence which is identical to $\beta$.
We say $\alpha$ {\em avoids} $\beta$ whenever $\alpha$ does not contain $\beta$.
For any binary sequences $\beta_1,\ldots,\beta_k$, we write $B_n(\beta_1,\ldots,\beta_k)$ to denote the set of binary sequences of length $n$ which avoid each of $\beta_1,\ldots, \beta_k$.
\end{definition}

As an example of Definition \ref{defn:bsavoid}, we note that $1111$ does not contain $000$.

As we show next, there is a simple formula for $B_n(\beta)$ which depends only on the length of $\beta$.

\begin{proposition}
\label{prop:binaryavoidcount}
Suppose $\beta$ is a binary sequence of length $k \ge 1$.
Then the number of binary sequences of length $n \ge 0$ which avoid $\beta$ is
\begin{equation}
\label{eqn:Bnbetacount}
|B_n(\beta)| = \sum_{j=0}^{k-1} \binom{n}{j}.
\end{equation}
The generating function for this sequence is 
\begin{equation}
\label{eqn:Bnbetaogf}
\sum_{n=0}^\infty |B_n(\beta)| x^n = \frac{(1-x)^k-x^k}{(1-2x)(1-x)^k}.
\end{equation}
\end{proposition}
\begin{proof}
Taking complements if necessary, we can assume without loss of generality that $\beta$ begins with $0$.
We argue by induction on $n + k$.

When $n+k = 1$ we have $k = 1$ and $n = 0$, so $\beta = 0$ and $B_0(\beta)$ contains only the empty sequence.
Since the right side of \eqref{eqn:Bnbetacount} is equal to $1$, the result holds in this case.

Now suppose $n + k \ge 2$.
When $k = 1$ we have $\beta = 0$, so $B_n(\beta) = \{\underbrace{1\cdots 1}_n\}$ and $|B_n(\beta)| = 1$.
The right side of \eqref{eqn:Bnbetacount} is also equal to $1$ in this case, so the result holds.
When $k \ge 2$ we consider two types of $\alpha \in B_n(\beta)$:  those $\alpha$ beginning with $1$ and those beginning with $0$.
We can construct each $\alpha \in B_n(\beta)$ which begins with $1$ uniquely by choosing a binary sequence in $B_{n-1}(\beta)$ and placing a $1$ at its left end.
By induction, the number of $\alpha \in B_n(\beta)$ which begin with $1$ is
\[ \sum_{j=0}^{k-1} \binom{n-1}{j}. \]
To construct those $\alpha \in B_n(\beta)$ which begin with $0$, first let $\beta_0$ be $\beta$ with its leftmost zero removed.
Now we can construct each $\alpha \in B_n(\beta)$ which begins with $0$ uniquely by choosing a binary sequence in $B_{n-1}(\beta_0)$ and placing a $0$ at its left end.
By induction, the number of $\alpha \in B_n(\beta)$ which begin with $0$ is 
\[ \sum_{j=0}^{k-2} \binom{n-1}{j}. \]

Combining our enumerations, we find
\begin{align*}
|B_n(\beta)| &= \sum_{j=0}^{k-1} \binom{n-1}{j} + \sum_{j=0}^{k-2} \binom{n-1}{j} \\
&= 1 + \sum_{j=1}^{k-1} \left( \binom{n-1}{j} + \binom{n-1}{j-1}\right) \\
&= \sum_{j=0}^{k-1} \binom{n}{j},
\end{align*}
which is what we wanted to prove.

To obtain the generating function for this sequence, we have
\begin{align*}
\sum_{n=0}^\infty \sum_{j=0}^{k-1} \binom{n}{j} x^n &= \sum_{j=0}^{k-1} \sum_{n=j}^\infty \binom{n}{j} x^n \\
&= \sum_{j=0}^{k-1} \frac{x^j}{(1-x)^{j+1}} \\
&= \frac{1}{1-x} \left( \frac{1-\left(\frac{x}{1-x}\right)^k}{1-\frac{x}{1-x}}\right) \\
&= \frac{(1-x)^k-x^k}{(1-2x)(1-x)^k}.
\end{align*}
\end{proof}

We can now use Proposition \ref{prop:binaryavoidcount} to enumerate $A_n(012,\beta)$.
We note that if $\beta$ contains $012$ then $A_n(012,\beta) = A_n(012)$, so we need only consider the case in which $\beta$ avoids $012$.
In addition, we are particularly interested in those $\beta$ with at least one $1$.

\begin{proposition}
\label{prop:An012beta}
Suppose $k \ge 1$ and $\beta \in A_k(012)$ contains at least one $1$.
Then for all $n \ge 1$ we have
\[ |A_n(012,\beta)| = \sum_{j = 0}^{k-2} \binom{n-1}{j}. \]
The generating function for this sequence is 
\[ \sum_{n=0}^\infty |A_n(012,\beta)| x^n = 1 + \frac{x(1-x)^{k-1} - x^k}{(1-2x)(1-x)^{k-1}}. \]
\end{proposition}
\begin{proof}
Since $\beta \in A_k(012)$ and $\beta$ contains at least one $1$, by Lemma \ref{lem:Fn123ascentchar} any copy of $\beta$ in $\alpha \in A_n(012)$ must be identical to $\beta$.
Furthermore, if $\alpha_0$ and $\beta_0$ are $\alpha$ and $\beta$, respectively, with their leftmost zeroes removed, then $\alpha \in A_n(012)$ avoids $\beta$ as a binary sequence if and only if $\alpha_0$ avoids $\beta_0$.
Now the result follows from Proposition \ref{prop:binaryavoidcount}.
\end{proof}

We will find it useful to extend Proposition \ref{prop:An012beta} to certain binary sequences $\beta$ which are not ascent sequences, and in particular, to binary sequences that begin with 1 and contain at least one 0.

\begin{proposition}
\label{prop:An012binary}
Suppose $k \ge 2$ and $\beta$ is a binary sequence of length $k$ which begins with $1$ and contains at least one $0$.
Then for all $n \ge 0$ we have
\[ |A_n(012,\beta)| = \sum_{j=0}^{k-1} \binom{n-1}{j}. \]
The generating function for this sequence is 
\[ \sum_{n=0}^\infty |A_n(012, \beta)| x^n = 1 + \frac{x (1-x)^k-x^{k+1}}{(1-2x)(1-x)^k}. \]
\end{proposition}
\begin{proof}
This is similar to the proof of Proposition \ref{prop:An012beta}.
\end{proof}

Some of the sequences in Propositions \ref{prop:binaryavoidcount}, \ref{prop:An012beta}, and \ref{prop:An012binary} appear in several other contexts.
For example, when $k = 3$ in \eqref{eqn:Bnbetacount} we obtain sequence A000124 in the OEIS, which is known as the sequence of central polygonal numbers, or the lazy caterer's sequence.
This sequence has a variety of combinatorial and algebraic interpretations.
Among others, its terms give the largest number of regions one can divide the plane into with a given number of lines, as well as the order dimension of the (strong) Bruhat order on the finite Coxeter group $B_{n+1}$.
When $k = 4$ in \eqref{eqn:Bnbetacount} we obtain sequence A000125 in the OEIS, which gives the largest number of regions one can divide a cube into with a given number of planes, and when $k = 5$ in \eqref{eqn:Bnbetacount} we obtain sequence A000127 in the OEIS.

\section{Enumerating $F_n(123,\sigma)$}
\label{sec:Fn123sigma}

In this section we turn our attention to the problem of enumerating $F_n(123,\sigma)$.
To begin, we establish a connection between $F_n(123,\sigma)$ and $A_n(012,\fishmap(\sigma))$.

\begin{proposition}
\label{prop:FntoAn}
Suppose $k \ge 1$, $\sigma \in F_k(123)$, and $\sigma \neq k \cdots 2 1$.
Then $\pi \in F_n(123,\sigma)$ if and only if $g(\pi) \in A_n(012, \fishmap(\sigma))$.
\end{proposition}
\begin{proof}
First note that by Lemma \ref{lem:piA123bijection} it's sufficient to show that $\pi \in F_n(123)$ avoids $\sigma$ if and only if $\fishmap(\pi)$ avoids $\fishmap(\sigma)$.

To start, observe that $\fishmap(n \cdots 21) = \underbrace{0\cdots 0}_n$ contains no $1$s, so it avoids $\sigma$, and the result holds for $\pi = n \cdots 2 1$.

Now suppose $\pi \neq n \cdots 2 1$.
By Proposition \ref{prop:buildascentfromA} we know $\pi$ and $\sigma$ are each concatenations of two decreasing sequences, say $\pi = \pi_1 \pi_2$ and $\sigma = \sigma_1 \sigma_2$.
As a result, $\pi$ contains $\sigma$ if and only if $\pi_1$ contains $\sigma_1$ and $\pi_2$ contains $\sigma_2$, since the only ascent in $\sigma$ must cross the only ascent in $\pi$.
By Proposition \ref{prop:buildascentfromA} the entries of $\fishmap(\pi)$ which correspond to the copy of $\sigma$ are $1$ or $0$ exactly according to whether the corresponding entries of $\fishmap(\sigma)$ are $1$ or $0$.
That is, $\pi$ contains $\sigma$ if and only if $\fishmap(\pi)$ contains $\fishmap(\sigma)$, which is what we wanted to prove.
\end{proof}

Proposition \ref{prop:FntoAn} enables us to use Proposition \ref{prop:An012beta} to enumerate $F_n(123,\sigma)$ for $\sigma \in F_k(123)$, assuming $\sigma$ is not monotonically decreasing.

\begin{proposition}
\label{prop:Fn123sigma}
Suppose $k \ge 1$, $\sigma \in F_k(123)$, and $\sigma \neq k \cdots 2 1$.
Then for all $n \ge 0$ we have
\[ |F_n(123,\sigma)| = \sum_{j=0}^{k-2} \binom{n-1}{j}. \]
The generating function for this sequence is 
\[ \sum_{n=0}^\infty |F_n(123,\sigma)| x^n = 1 + \frac{x(1-x)^{k-1} - x^k}{(1-2x)(1-x)^{k-1}}. \]
\end{proposition}
\begin{proof}
By Proposition \ref{prop:FntoAn} we have $|F_n(123,\sigma)| = |A_n(012,\fishmap(\sigma))|$ for all $n \ge 0$.
Now the result follows from Proposition \ref{prop:An012beta}.
\end{proof}

One consequence of Proposition \ref{prop:Fn123sigma} is that when $\sigma \in F_k(123)$ and $\sigma$ is not monotonically decreasing, the quantity $|F_n(123,\sigma)|$ depends on $k$ but not on $\sigma$.
We note that this property will not usually hold if we enumerate $F_n(123,\sigma)$ with respect to a statistic like inversion number or number of left-to-right maxima.
In particular, if $|\sigma| = k$ then we will see differences in distribution on $F_k(123,\sigma)$ as long as our statistic is not constant on $F_k(123)$.
On the other hand, one can show that a similar independence result does hold in the classical case:  if $\sigma \in S_k(231, 123)$ is not monotone, then the generating function for $|S_n(231, 123, \sigma)|$ is given by
\[ \frac{1-2x+2x^2-x^k}{(1-x)^3} = \frac{x^{k-1} + x^{k-2} + x^{k-3} + \cdots + x^2 - x + 1}{(1-x)^2}. \]
In particular, $|S_n(231, 123, \sigma)|$ depends on $k$ but not on $\sigma$.
Taking this a bit further, one can show that the permutations which avoid $231$ and $123$ are the permutations we obtain by replacing each entry of $312$ with a decreasing sequence of consecutive positive integers.
In the terminology of Albert, Atkinson, Bouvel, Ru\v{s}kuc, and Vatter \cite{AABRV}, these permutations form the geometric grid class associated with the diagram in Figure \ref{fig:213grid}.
\begin{figure}[ht]
\centering
\includegraphics[width=1in]{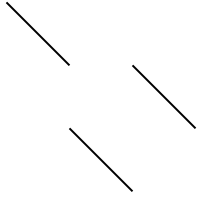}
\caption{The grid for $S_n(123,231)$.}
\label{fig:213grid}
\end{figure}
We note that we can replace $312$ with another permutation $\tau$ to obtain another grid class.
It is an open question to determine for which $\tau$ the number of permutations in the class which avoid a permutation $\sigma$ in the class depends on $|\sigma|$ but not on $\sigma$.

In investigating $F_n(123,\sigma)$, we do not need to consider any $\sigma$ which contains 123, since any permutation which avoids 123 will necessarily also avoid that $\sigma$. 
In other words, if $\sigma$ contains 123 then $F_n(123,\sigma) = F_n(\sigma)$. 
However, this is not the case for $\sigma$ which contains $\fish$. 
That is, there exist $\sigma$ which avoid $123$ and contain $\fish$, but for which $F_n(123,\sigma) \neq F_n(123)$ in general.
Next we characterize those $\sigma$.

\begin{definition}
We say a permutation $\sigma$ is {\em unrestrictive} whenever $\sigma$ avoids $123$, $\sigma$ contains $\fish$, and $\sigma$ contains at least one of $2413$ and $3412$.
We say $\sigma$ is {\em restrictive} whenever $\sigma$ avoids $123$, $\sigma$ contains $\fish$, and $\sigma$ avoids both $2413$ and $3412$.
\end{definition}

As their name suggests, and as we show next, unrestrictive permutations do not introduce additional restrictions on $F_n(123)$.

\begin{lemma}
\label{lem:unrestrictive}
Suppose $\sigma$ is an unrestrictive permutation.
If $\pi$ avoids $\fish$ and $\pi$ avoids $123$ then $\pi$ avoids $\sigma$.
\end{lemma}
\begin{proof}
We first prove the result for $\sigma = 2413$ and $\sigma = 3412$.

Suppose by way of contradiction that $\pi$ avoids $\fish$ and $123$ and $\pi$ contains $2413$.
Choose a copy $abcd$ of $2413$ in $\pi$ in which $b$ is as large as possible, and let $e$ be the entry of $\pi$ immediately to the left of $b$.
Note that we might have $e = a$.
If $e > b$ then $e \neq a$ (since $a < b$) and $aecd$ is a copy of $2413$ with a larger second entry.
This contradicts our choice of $b$, so we must have $e < b$.
On the other hand, if $e = 1$ then $ecd$ is a copy of $123$, so $e > 1$.
Now if $e - 1$ is to the left of $e$ then $e-1, e, b$ is a copy of $123$.
But if $e - 1$ is to the right of $e$ then it's also to the right of $b$ by our choice of $e$, and $e, b, e-1$ is a copy of $\fish$.
This exhausts all possibilities, so if $\pi$ avoids $\fish$ and $\pi$ avoids $123$ then it also avoids $2413$.

Note that in the argument above we did not use the fact that $a < d$, so the same argument shows that if $\pi$ avoids $\fish$ and $\pi$ avoids $123$ then it also avoids $3412$.

Finally, since $\sigma$ is unrestrictive it must contain at least one of $2413$ and $3412$.
Now the result follows.
\end{proof}

\begin{proposition}
If $\sigma$ is an unrestrictive permutation then $F_n(123,\sigma) = F_n(123)$ for all $n \ge 0$.
In particular, $|F_n(123,\sigma)| = 2^{n-1}$ for all $n \ge 1$.
\end{proposition}
\begin{proof}
We know $F_n(123,\sigma) \subseteq F_n(123)$, so it's sufficient to show $F_n(123) \subseteq F_n(123, \sigma)$.
But this follows from Lemma \ref{lem:unrestrictive}.
\end{proof}

Turning our attention to restrictive permutations, we first show that these permutations cannot contain more than one copy of $\fish$.

\begin{lemma}
\label{lem:justonef}
Suppose $\sigma$ is a restrictive permutation.
Then $\sigma$ contains exactly one copy of $\fish$.
\end{lemma}
\begin{proof}
Suppose $\sigma$ contains two or more copies of $\fish$.
If $a_1 b_1 c_1$ and $a_2 b_2 c_2$ are two copies of $\fish$ then the subsequence formed by the union of their entries has length at most $6$.
This subsequence corresponds to a restrictive permutation of length at most $6$ with two or more copies of $\fish$.
But a computer search shows no such permutations exist.
\end{proof}

Lemma \ref{lem:justonef} allows us to describe how to transform a restrictive permutation into a Fishburn permutation by adding a single entry.
This will enable us to use Proposition \ref{prop:Fn123sigma} to enumerate $F_n(123,\sigma)$ when $\sigma$ is restrictive.

\begin{proposition}
Suppose $k \ge 3$ and $\sigma \in S_k$ is a restrictive permutation.
Then for all $n \ge 0$ we have
\[ |F_n(123, \sigma)| = \sum_{j=0}^{k-1} \binom{n-1}{j}. \]
\end{proposition}
\begin{proof}
By Lemma \ref{lem:justonef}, we know $\sigma$ contains exactly one copy of $\fish$, whose entries we write as $a, b, a-1$.
Note $a$ and $b$ are adjacent and $a \ge 2$.
Let $\tau$ be the permutation of length $k+1$ we obtain by adding $1$ to each entry of $\sigma$ and then inserting $1$ between $a+1$ and $b+1$.
We claim $\tau$ avoids both $\fish$ and $123$.

To show $\tau$ avoids $\fish$, suppose by way of contradiction that $x+1, y+1, x$ is a copy of $\fish$ in $\tau$.
If $x \neq 1$ then $x, y, x-1$ is a copy of $\fish$ in $\sigma$ that is different from $a, b, a-1$, which is a contradiction.
Therefore, $x = 1$ and our copy of $\fish$ in $\tau$ is $2, y+1, 1$.
Since the $1$ is between $a+1$ and $b+1$ in $\tau$, this means the $1$ in $\sigma$ is to the left of $a$.
Now $1, a, b$ is a copy of $123$ in $\sigma$, which is a contradiction.
Therefore, $\tau$ avoids $\fish$.

To show $\tau$ avoids $123$, suppose by way of contradiction that $x+1, y+1, z+1$ is a copy of $123$ in $\tau$.
If $x \neq 0$ then $xyz$ is a copy of $123$ in $\sigma$, so our copy of $123$ in $\tau$ must be $1, y+1, z+1$.
Now if $y = b$ then $abz$ is a copy of $123$ in $\sigma$, which is a contradiction.
Since $1$ is between $a+1$ and $b+1$ in $\tau$, it must be the case that $y$ is to the right of $b$ in $\sigma$.
In addition, if $z > b$ then $abz$ is also a copy of $123$ in $\sigma$.
So we must have $y < z < b$, with $y$ and $z$ both to the right of $b$ in $\sigma$.
In other words, $abyz$ is a copy of either $2413$ or $3412$ in $\sigma$, depending on whether $a < z$ or $a > z$.
But this contradicts the fact that $\sigma$ avoids $2413$ and $3412$.

We now claim that $F_n(123,\sigma) = F_n(123, \tau)$ for all $n \ge 0$.

To prove this, first note that since $\sigma$ is contained in $\tau$, we have $F_n(123,\sigma) \subseteq F_n(123,\tau)$.
Now suppose $\pi \in F_n(123)$ contains $\sigma$.
We show $\pi$ also contains $\tau$.

Choose a copy of $\sigma$ in $\pi$ and let $x$, $y$, and $z$ be the entries of $\pi$ playing the roles of $a$, $b$, and $a-1$, respectively.
We claim there is an entry of $\pi$ between $x$ and $y$ which is less than $x$.
If not, then $x-1$ would either be to the left of $x$ or to the right of $y$.
In the first case $x-1, x, y$ is a copy of $123$, which $\pi$ avoids.
In the second case, if $u$ is the entry immediately to the right of $x$ then $x, u, x-1$ is a copy of $\fish$, which $\pi$ also avoids.

We now claim that in fact $1$ is between $x$ and $y$.
To see this, let $v$ be the smallest entry between $x$ and $y$.
Let $w$ be the entry immediately to the right of $v$ and note that $w > v$.
Now suppose $v > 1$.
If $v-1$ is to the left of $v$ then $v-1, v, w$ is a copy of $123$, which $\pi$ avoids.
And if $v-1$ is to the right of $v$ then it's also to the right of $w$ and $v, w, v-1$ is a copy of $\fish$, which $\pi$ also avoids.
This exhausts all possibilities, so we must have $v = 1$.

Since $1$ appears between $x$ and $y$ in $\pi$, we can include it in our copy of $\sigma$ to obtain a copy of $\tau$.
\end{proof}

\section{Enumerating $F_n(321,1423)$}
\label{sec:Fn3211432}

Having enumerated $F_n(123,\sigma)$ for all $\sigma$, we now turn our attention to $F_n(321,\sigma)$.
In this case one can obtain a wider variety of enumerative results, because these enumerations depend on $\sigma$, instead of just $|\sigma|$.
In the next three sections we use generating trees to study three cases with particularly nice enumerations.
(For background on generating trees and their application to pattern-avoiding permutations see \cite{West}.)
In each case we begin by considering generating functions with respect to inversion number and number of left-to-right maxima before specializing to obtain our enumeration of $F_n(321,\sigma)$.
We start with the generating tree structure for Fishburn permutations which avoid 321 and 1423.

\begin{lemma}
\label{lem:ndescenttop}
Suppose $n \ge 2$ and $\pi \in F_n(321, 1423)$ has a descent.
Then either $n$ is the left entry of the rightmost descent or $\pi_n = n$.
\end{lemma}
\begin{proof}
If $\pi_n \neq n$ then $n$ is the left entry of a descent.
If this descent is not the rightmost descent, then $n$ and the rightmost descent form a copy of $321$.
Since $\pi$ avoids $321$, the result follows.
\end{proof}

Lemma \ref{lem:ndescenttop} enables us to classify the permutations in $F_n(321,1423)$.

\begin{proposition}
\label{prop:F3211423labels}
Suppose $n \ge 1$ and $\pi \in F_n(321, 1423)$.
Then exactly one of the following holds.
\begin{enumerate}
\item[{\upshape (1)}]
$n \ge 2$, $\pi_{n-1} = n$, and $\pi_n = n-1$.
\item[{\upshape (2a)}]
$\pi = 1$.
\item[{\upshape (2b)}]
$n \ge 4$, $\pi_{n-1} = n$, and $\pi_n = k$, where $1 < k < n-1$.
\item[{\upshape (2c)}]
$n \ge 3$ and the rightmost $k+2$ entries of $\pi$ are $n,1,2,\ldots,k,n-1$ for some $k$ with $1 \le k \le n-2$.
\item[{\upshape (2d)}]
$\pi_n = n$ and $\pi$ has a descent.
\item[{\upshape (3)}]
$n \ge 2$ and $\pi$ has no descents.
\end{enumerate}
\end{proposition}
\begin{proof}
First note that the six cases are mutually exclusive, so it's sufficient to show that $\pi$ belongs to at least one of them.

If $n = 1$ then $\pi = 1$ and $\pi$ is in case (2a), so we can assume $n \ge 2$.
Now if $\pi$ has no descents then $\pi$ is in case (3), so we can assume $\pi$ has a descent.
And if $\pi_n = n$ then $\pi$ is in case (2d), so we can assume $\pi_n \neq n$.

By Lemma \ref{lem:ndescenttop}, the left entry of the rightmost descent in $\pi$ is $n$.
If $n = 2$ then we are in case (1), so we can assume $n \ge 3$.
If $\pi_{n-1} = n$ and $\pi_n = 1$ then 2, the entry immediately to the right of 2, and 1 form a copy of $\fish$.
Therefore, if $\pi_{n-1} = n$ then $\pi_n > 1$.
If $\pi_n = n-1$ then we are in case (1) and if $\pi_n < n-1$ then we are in case (2b).

Now suppose $\ell \ge 2$ and the entries to the right of $n$ are $a_1, a_2, \ldots, a_\ell$.
Since $\pi$ avoids $321$, we must have $a_1 < a_2 < \cdots < a_\ell$.
If 1 is to the left of $n$ then $1,n,a_1,a_2$ is a copy of $1423$, so 1 must be to the right of $n$ and $a_1 = 1$.
Similarly, $a_j = j$ for $1 \le j \le \ell-1$.
Now if $a_\ell < n-1$ then $a_\ell + 1$ is the smallest number to the left of $n$.
Therefore, $a_\ell + 1$, the entry immediately to its right, and $a_\ell$ form a copy of $\fish$.
Combining all of this, we see that if $\ell \ge 2$ then we are in case (2c).
\end{proof}

We can use Proposition \ref{prop:F3211423labels} to describe the generating tree for Fishburn permutations which avoid $321$ and $1423$.
For the rest of this section, when we say $\pi \in F_n(321,1423)$ has label (x), we mean $\pi$ belongs to case (x) in Proposition \ref{prop:F3211423labels}.

In our next result we describe the active sites in each $\pi \in F_n(321, 1423)$ according to $\pi$'s label.
Inserting $n+1$ in an active site produces a permutation $\gamma \in F_{n+1}(321, 1423)$;  for each active site we give $\gamma$'s label.
To do this, we label each active site in $\pi$ with the label of the resulting permutation.
For example, when we write $2\ 1^{1} 3^{2e}$ we mean $213$ has two active sites, inserting $4$ into the left active site produces a permutation with label (1), and inserting $4$ into the right active site produces a permutation with label (2e).

\begin{proposition}
\label{prop:F3211423rules}
Suppose $n \ge 1$ and $\pi \in F_n(321, 1423)$.
Then we have the following.
\begin{enumerate}
\item[{\upshape (1)}]
If $\pi$ has label (1) then $\pi = \cdots n\ n-1^{2d}$.
\item[{\upshape (2a)}]
If $\pi$ has label (2a) then $\pi =\,  ^11^3$.
\item[{\upshape (2b)}]
If $\pi$ has label (2b) then $\pi = \cdots n^{2b} k^{2d}$.
\item[{\upshape (2c)}]
If $\pi$ has label (2c) then $\pi = \cdots k^{2b} n-1^{2d}$.
\item[{\upshape (2d)}]
If $\pi$ has label (2d) then $\pi = \cdots\, ^1n^{2d}$.
\item[{\upshape (3)}]
If $\pi$ has label (3) then $\pi = \, ^{2c}1 \cdots n-1^1 n^{3}$.
\end{enumerate}
\end{proposition}
\begin{proof}
(1)
If $\pi$ has label (1) and we insert $n+1$ into any site to the left of $n$ then $n+1, n, n-1$ will be a copy of $321$.
And if we insert $n+1$ between $n$ and $n-1$ then $n, n+1, n-1$ will be a copy of $\fish$.
Therefore, only the rightmost site can be active.
On the other hand, none of our forbidden patterns ends with its largest element, so the rightmost site is active.
We can check that inserting $n+1$ into that site produces a permutation with label (2d).

(2a)
If $\pi = 1$ then both sites are active, since all of our forbidden patterns have length at least three.
We can check that inserting 2 in the right site produces a permutation with label (3) and inserting 2 in the right site produces a permutation with label (1).

(2b)
As in (1), if we insert $n+1$ in any site to the left of $n$ then $n+1, n, k$ will be a copy of $321$, so none of these sites is active.
If we insert $n+1$ in either of these sites then we cannot obtain a copy of $321$ or $1423$, because the $n+1$ would play the role of the largest element in any such copy, and $n+1$ is too far to the right to do that.
And because $k < n-1$, inserting $n+1$ in either of these sites does not result in a copy of $\fish$.
Therefore both sites are active.
We can check that inserting $n+1$ into the left site produces a permutation with label (2b) and inserting $n+1$ into the right site produces a permutation with label (2d).

(2c)
If we insert $n+1$ anywhere to the right of $n$ and to the left of $k$, we produce a copy of $1423$, and if we insert $n+1$ anywhere to the left of $n$ then we produce a copy of $321$.
Therefore, only the rightmost two sites can be active.
As in (2b), inserting $n+1$ into either of these sites cannot produce a copy of $321$ or $1423$.
And it cannot produce a copy of $\fish$ because $k < n-1$.
It follows that the rightmost two sites are, in fact, active.
We can check that inserting $n+1$ into the left site produces a permutation with label (2b) and inserting $n+1$ into the right site produces a permutation with label (2d).

(2d)
As in (2c), if we insert $n+1$ anywhere to the right of the rightmost descent and to the left of $\pi_{n-1}$ we produce a copy of $1423$, and if we insert $n+1$ anywhere to the left of the rightmost descent then we produce a copy of $321$.

To show the site within the rightmost descent is not active, let the left entry of that descent be $b$ and let the entries between $b$ and $n$ be $a_1 < \cdots <a_\ell$.
If we remove $n$ from $\pi$ then we obtain a permutation in $F_{n-1}(321,1423)$ which does not end with $n-1$, so by Lemma \ref{lem:ndescenttop} we must have $b = n-1$ and $a_\ell < b$.
If $a_1 \neq 1$ then inserting $n+1$ produces a copy of $1423$, namely $1, n+1, a_1, n$, so the site can only be active if $a_1 = 1$.
Similarly, the site can only be active if $a_j = j$ for $1 \le j \le \ell$, so $\pi = \cdots n-1,1,2,3,\ldots,\ell,n$.
If $\ell = n-2$ then inserting $n+1$ in the rightmost descent produces a copy of $\fish$, namely $n-1, n+1, n-2$.
On the other hand, if $\ell < n-2$ then $\ell +1$, the entry immediately to its right, and $\ell$ is a copy of $\fish$.
Therefore, the site in the rightmost descent cannot be active.

We can now check that the rightmost two sites are active, inserting $n+1$ into the left site produces a permutation with label (1), and inserting $n+1$ into the right site produces a permutation with label (2d).

(3)
We can check that inserting $n+1$ into any site other than the leftmost or one of the two rightmost will produce a copy of $1423$, but inserting into any of these three sites produces no forbidden pattern.
We can also check that the remaining three sites are all active and inserting $n+1$ into one of them produces a permutation with the given label.
\end{proof}

Taken together, Propositions \ref{prop:F3211423labels} and \ref{prop:F3211423rules} show that the generating tree for the Fishburn permutations which avoid $321$ and $1423$ has the following description.

\begin{proposition}
\label{prop:F3211423gentree}
The generating tree for the Fishburn permutations which avoid $321$ and $1423$ is given by the following.

\medskip

\begin{tabular}{lll}
Root: & & $(2x)$ \\
& & \\
Rules: & & $(1) \rightarrow (2z)$  \\[2ex]
& & $(2x) \rightarrow (1), (3)$ \\[2ex]
& & $(2y) \rightarrow (2y), (2z)$ \\[2ex]
& & $(2z) \rightarrow (1), (2z)$ \\[2ex]
& & $(3) \rightarrow (2y), (1), (3)$
\end{tabular}
\end{proposition}
\begin{proof}
When we write down the transition rules implied by Proposition \ref{prop:F3211423rules}, we note that (2b) and (2c) have the same children.
Therefore we can identify the labels (2b) and (2c).
We call this new label (2y), we replace (2a) with (2x), and we replace (2d) with (2z) to obtain the given generating tree.
\end{proof}

We can use Proposition \ref{prop:F3211423gentree} to find formulas for the number of permutations in $F_n(321, 1423)$ with each possible label, as well as to enumerate $F_n(321, 1423)$.
However, we will obtain more general results by finding the generating functions for these permutations with respect to the number of inversions and the number of left-to-right maxima they contain.
We begin by setting some notation.

\begin{definition}
\label{defn:F3211423notation}
For all $n \ge 0$ and each label (y) listed in Proposition \ref{prop:F3211423labels}, we write $[y]_n(q,t)$ to denote the generating function given by
\[ [y]_n(q,t) =  \sum_{\substack{\pi \in F_n(321,1432) \\ \lbl(\pi) = (y)}} q^{\inv(\pi)} t^{\ltormax(\pi)} \]
and we write $[y](q,t,x)$ to denote the generating function given by
\[ [y](q,t,x) = \sum_{n=0}^\infty [y]_n(q,t) x^n. \]
We sometimes abbreviate $[y]_n = [y]_n(q,t)$ and $[y] = [y](q,t,x)$.
We also write $T_n(q,t)$ to denote the generating function given by
\[ T_n(q,t) = \sum_{\pi \in F_n(321, 1423)} q^{\inv(\pi)} t^{\ltormax(\pi)} \]
and we write $T(q,t,x)$ to denote the generating function given by 
\[ T(q,t,x) = \sum_{n=0}^\infty T_n(q,t) x^n.\]
\end{definition}

Proposition \ref{prop:F3211423labels} gives us a simple formula for $[2a]_n$.

\begin{proposition}
Using the notation in Definition \ref{defn:F3211423notation}, for all $n \ge 1$ we have
\begin{equation}
\label{eqn:2aogf}
[2a]_n = \begin{cases} t & \text{if\ $n=1$;} \\ 0 & \text{if\ $n \ge 2$.} \end{cases}
\end{equation}
\end{proposition}
\begin{proof}
This follows from Proposition \ref{prop:F3211423labels}, since $1$ is the only permutation with label (1a).
\end{proof}

As we show next, Proposition \ref{prop:F3211423rules} gives us recurrence relations for $[y]_n$ for each label (y) in Proposition \ref{prop:F3211423labels}.

\begin{proposition}
\label{prop:F3211423recurrences}
Using the notation in Definition \ref{defn:F3211423notation}, for all $n \ge 2$ we have the following.
\begin{enumerate}
\item[{\upshape (i)}]
\begin{equation}
\label{eqn:1recurrence}
[1]_n = q [2a]_{n-1} + q [2d]_{n-1} + q [3]_{n-1}.
\end{equation}
\item[{\upshape (ii)}]
\begin{equation}
\label{eqn:2brecurrence}
[2b]_n = qt[2b]_{n-1} + q [2c]_{n-1}.
\end{equation}
\item[{\upshape (iii)}]
\begin{equation}
\label{eqn:2crecurrence}
[2c]_n = q^{n-1} t^{2-n} [3]_{n-1}.
\end{equation}
\item[{\upshape (iv)}]
\begin{equation}
\label{eqn:2drecurrence}
[2d]_n = t [1]_{n-1} + t [2b]_{n-1} + t [2c]_{n-1} + t [2d]_{n-1}.
\end{equation}
\item[{\upshape (v)}]
\begin{equation}
\label{eqn:3recurrence}
[3]_n = t [2a]_{n-1} + t [3]_{n-1}.
\end{equation}
\end{enumerate}
\end{proposition}
\begin{proof}
(i)
By Proposition \ref{prop:F3211423rules}, we can obtain a permutation in $F_n(321,1423)$ with label (1) in one of three ways: insert a 2 to the left of the 1 in a permutation with label (2a), insert an $n$ immediately to the left of the rightmost entry in a permutation of length $n-1$ with label (2d), or insert an $n$ immediately to the left of the rightmost entry of a permutation with label (3).
In all three cases we increase the number of inversions by one and leave the number of left-to-right maxima unchanged, and the result follows.

(ii)
This is similar to the proof of (i).

(iii)
By Proposition \ref{prop:F3211423rules}, we can only obtain a permutation in $F_n(321,1423)$ with label (2c) by inserting $n$ at the left end of $1 2 \cdots n-1$, which has label (3).
This increases the number of inversions by $n-1$ and decreases the number of left-to-right maxima by $n-2$, and the result follows.

(iv),(v)
These are similar to the proof of (i).
\end{proof}

We can use the recurrence relations in Proposition \ref{prop:F3211423recurrences} to obtain simple formulas for $[2b]_n$, $[2c]_n$, and $[3]_n$.

\begin{proposition}
Using the notation in Definition \ref{defn:F3211423notation}, for all $n \ge 0$ we have the following.
\begin{enumerate}
\item[{\upshape (i)}]
\begin{equation}
\label{eqn:2bnformula}
[2b]_n = \begin{cases} 0 & \text{if\ $n < 3$;} \\ q^{n-1} t \left(\frac{t^{n-3}-1}{t-1}\right) & \text{if $n \ge 3$.}\end{cases}
\end{equation}
\item[{\upshape (ii)}]
\begin{equation}
\label{eqn:2cnformula}
[2c]_n = \begin{cases} 0 & \text{if\ $n < 3$;} \\ q^{n-1} t & \text{if\ $n \ge 3$.} \end{cases}
\end{equation}
\item[{\upshape (iii)}]
\begin{equation}
\label{eqn:3nformula}
[3]_n = \begin{cases} 0 & \text{if\ $n < 2$;} \\ t^n & \text{if $n \ge 2$.} \end{cases}
\end{equation}
\end{enumerate}
\end{proposition}
\begin{proof}
We can check these results directly for $n \le 4$, so suppose $n \ge 5$;  we argue by induction on $n$.

(iii)
This follows from \eqref{eqn:3recurrence}, \eqref{eqn:2aogf}, and induction.

(ii)
This follows from \eqref{eqn:2crecurrence} and (iii).

(i)
This follows from \eqref{eqn:2brecurrence}, (ii), and induction.
\end{proof}

We also have expressions for the generating functions $[1]$, $[2b]$, $[2c]$, $[2d]$, and $[3]$.

\begin{proposition}
Using the notation in Definition \ref{defn:F3211423notation}, we have the following.
\begin{enumerate}
\item[{\upshape (i)}]
\begin{equation}
\label{eqn:1formula}
[1] = \frac{qt x^2-q^2 t(1+t) x^3 + q^3 t^2 x^4 + q^3 t^2 x^5 + q^4 t^2 x^6-q^4t^3 x^6}{(1-qx)(1-qtx)(1-tx-qtx^2)}.
\end{equation}
\item[{\upshape (ii)}]
\begin{equation}
\label{eqn:2bformula}
[2b] = \frac{q^3 t x^4}{(1-qx)(1-qtx)}.
\end{equation}
\item[{\upshape (iii)}]
\begin{equation}
\label{eqn:2cformula}
[2c] = \frac{q^2 t x^3}{1-qx}.
\end{equation}
\item[{\upshape (iv)}]
\begin{equation}
\label{eqn:2dformula}
[2d] = \frac{qt^2 x^3 - q^2 t^3 x^4 + q^2 t^2(q-t) x^5 + q^3 t^3 (t-1) x^6}{(1-qx)(1-tx)(1-qtx)(1-tx-qtx^2)}.
\end{equation}
\item[{\upshape (v)}]
\begin{equation}
\label{eqn:3formula}
[3] = \frac{t^2 x^2}{1-tx}.
\end{equation}
\end{enumerate}
\end{proposition}
\begin{proof}
(v)
Multiply \eqref{eqn:3nformula} by $x^n$ and sum over $n \ge 2$.

(iii)
Multiply \eqref{eqn:2cnformula} by $x^n$ and sum over $n \ge 3$.

(ii)
Multiply \eqref{eqn:2bnformula} by $x^n$ and sum over $n \ge 3$.

(iv)
When we replace $n$ with $n-1$ in \eqref{eqn:1recurrence} and use the result to eliminate $[1]_{n-1}$ in \eqref{eqn:2drecurrence} we find that for $n \ge 3$ we have
\[ [2d]_n = t [2d]_{n-1} + qt [2d]_{n-2} + qt[2a]_{n-2} + t [2b]_{n-1} + t [2c]_{n-1} + qt [3]_{n-2}. \]
When we multiply this by $x^n$ and sum over $n \ge 3$, we find
\[ [2d] = t x [2d] + qtx^2 [2d] + qt^2 x^3+ tx[2b]+tx[2c]+qtx^2[3]. \]
Now we can use (ii), (iii), and (v), respectively, to eliminate $[2b]$, $[2c]$, and $[3]$, and solve the resulting equation for $[2d]$, to obtain (iv).

(i)
This is similar to the proof of (iv).
\end{proof}

We note that $[2d](q,t,x)$ can be rewritten as
\[ \frac{qt^2 x^3}{1-tx-qtx^2} + \frac{q^3t^2x^5}{(1-qx)(1-qtx)(1-tx-qtx^2)} + \frac{q^2 t^2 x^4}{(1-qx)(1-tx-qtx^2)} + \frac{qt^3 x^4}{(1-tx)(1-tx-qtx^2)} \]
and $[1](q,t,x)$ can be rewritten as
\begin{align*}
qt x^2 + \frac{q^2 t^2 x^4}{1-tx-qtx^2} &+ \frac{q^4 t^2 x^6}{(1-qx)(1-qtx)(1-tx-qtx^2)} \\ 
&+ \frac{q^3 t^2 x^5}{(1-qx)(1-tx-qtx^2)} + \frac{q^2 t^3 x^5}{(1-tx)(1-tx-qtx^2)} + \frac{qt^2x^3}{1-tx}.
\end{align*}

Next we use our expressions for $[1]$, $[2b]$, $[2c]$, $[2d]$, and $[3]$ to obtain an expression for $T(q,t,x)$.

\begin{proposition}
Using the notation in Definition \ref{defn:F3211423notation}, we have 
\begin{equation}
\label{eqn:Tformula}
T(q,t,x) = \frac{1-q(1+t)x + tq^2 x^2 + q^2tx^3+q^3t(1-t)x^4}{(1-qx)(1-qtx)(1-tx-qtx^2)}.
\end{equation}
\end{proposition}
\begin{proof}
Combine the fact that
\[ T(q,t,x) = 1 + [1] + [2a] + [2b] + [2c] + [2d] + [3] \]
with \eqref{eqn:1formula}, \eqref{eqn:2aogf}, \eqref{eqn:2bformula}, \eqref{eqn:2cformula}, \eqref{eqn:2dformula}, and \eqref{eqn:3formula} and simplify the result.
\end{proof}

We can now use our generating functions to obtain $|F_n(321,1423)|$, as well as to find the number of permutations in $F_n(321,1423)$ with each label.
It will be helpful to introduce notation for these numbers.

\begin{definition}
\label{defn:3211423enumerationnotation}
For each label (y) and each $n \ge 0$, we write $(y)_n$ to denote the number of permutations in $F_n(321,3124)$ with label (y).
We note that $(y)_n = [y]_n(1,1)$.
\end{definition}

Our enumerations will involve Fibonacci numbers;  recall that we index this sequence so that $F_0 = F_1 = 1$ and $F_n = F_{n-1} + F_{n-2}$ for $n \ge 2$.

\begin{proposition}
Using the notation in Definition \ref{defn:3211423enumerationnotation}, we have the following.
\begin{enumerate}
\item[{\upshape (i)}] For all $n \ge 2$, we have $(1)_n = F_n - n + 1$.
\item[{\upshape (ii)}] For all $n \ge 3$ we have $(2b)_n = n-3$.
\item[{\upshape (iii)}] For all $n \ge 3$ we have $(2c)_n = 1$.
\item[{\upshape (iv)}] For all $n \ge 0$ we have $(2d)_n = F_{n+1} - n - 1$.
\item[{\upshape (v)}] For all $n \ge 1$ we have $(3)_n = 1$.
\item[{\upshape (vi)}] For all $n \ge 0$ we have 
\begin{equation}
\label{eqn:F3211423enumeration}
|F_n(321, 1423)| = F_{n+2}-n-1.
\end{equation}
\end{enumerate}
\end{proposition}
\begin{proof}
Set $q = t = 1$ in \eqref{eqn:1formula}, \eqref{eqn:2bformula}, \eqref{eqn:2cformula}, \eqref{eqn:2dformula}, \eqref{eqn:3formula}, and \eqref{eqn:Tformula} and find the coefficient of $x^n$.
\end{proof}

We note that the sequence in \eqref{eqn:F3211423enumeration} is sequence A000126 in the OEIS with a leading 1 prepended.
The OEIS entry for this sequence includes several combinatorial interpretations, one of which leads to the following open problem.

\begin{openproblem}
\label{op:1423binseq}
For each $n \ge 1$ find a constructive bijection between $F_n(321, 1423)$ and the set of binary sequences of length $n -1$ in which no two consecutive $1$s have more than one $0$ between them.
\end{openproblem}

We also note that the terms in the sequence in \eqref{eqn:F3211423enumeration} are all one more than the terms in sequence A001924 in the OEIS.
Zhuang has shown \cite[Corollary 7]{YZM} that the terms in sequence A001924 count certain Motzkin paths, which leads to the following open problem.

\begin{openproblem}
\label{op:1423Motzkin}
For each $n \ge 1$, let $D_n$ be the set $F_n(321,1423)$ with the permutation $12 \cdots n$ removed.
Find a constructive bijection between $D_n$ and the set of Motzkin paths of length $n$ with exactly one ascent.
\end{openproblem}

Elizalde, Zhuang, and Troyka have also conjectured that the sequence A001924 also counts permutations of length $n$ which avoid the consecutive pattern 213 and whose inverses have exactly $n-2$ descents.
We note that when $n = 3$ both this set and the set of inverses of permutations in this set include $231$, so neither set is equal to $F_3(321,1423)$.

\section{Enumerating $F_n(321, 3124)$}
\label{sec:Fn3213124}

In this section we continue our study of $F_n(321,\sigma)$ by using generating trees to study the generating function for $F_n(321,3124)$.
As we will see, this set is equinumerous with $F_n(321,3124)$, but its generating tree has a substantially different description.
In particular, the generating tree for $F_n(321,1423)$ has finitely many labels while the generating tree for $F_n(321,3124)$ has infinitely many.
We begin by classifying the permutations in $F_n(321,3124)$.

\begin{proposition}
\label{prop:F3213124labels}
Suppose $n \ge 1$ and $\pi \in F_n(321, 3124)$.
Then exactly one of the following holds.
\begin{enumerate}
\item[{\upshape (1a)}]
$\pi_{n-1} = n$ and $\pi_n = n-1$.
\item[{\upshape (1b)}]
$\pi$ contains $312$.
\item[{\upshape (k)}]
$\pi$ avoids $312$ and there exists a unique $k \ge 2$ such that the rightmost $k$ entries of $\pi$ are $a,n-k+2, n-k+3,\ldots,n$, where $a \le n-k$.
\end{enumerate}
Note that (k) includes infinitely many cases:  (2), (3), etc.
\end{proposition}
\begin{proof}
We first claim these cases are mutually exclusive.

This follows for the various (k) since $k$ is the length of the longest increasing sequence of consecutive entries at the right end of $\pi$.
It follows for (k) and (1a) since (k) requires $\pi_n = n$ and (1a) requires $\pi_n = n-1$.

If $\pi$ satisfies (k) for some $k \ge 2$ and (1b), or if $\pi$ satisfies (1a) and (1b), then $n$ cannot be involved in any copy of 312 in $\pi$.
Therefore any such copy, combined with $n$, will produce a copy of $3124$.
This contradicts that fact that $\pi$ avoids $3124$.

Now suppose $\pi \in F_n(321, 3124)$;  it's sufficient to show $\pi$ satisfies at least one of the given conditions.

If $\pi_n = n$ then any copy of $312$ would produce a copy of $3124$, so $\pi$ avoids $312$.
Now $\pi$ satisfies (k) for some $k \ge 2$.

If $\pi_n < n$ and $\pi$ contains $312$ then $\pi$ satisfies (1b).

If $\pi_n < n$ and $\pi$ avoids $312$ then we must have $\pi_{n-1} = n$, since having $n$ any further to the left will produce a copy of $312$ or $321$.
Now suppose $\pi_{n-1} = \ell$ and, by way of contradiction, $\ell < n-1$.
Then $\ell + 1$ is somewhere to the left of $n$ in $\pi$.
If $\ell + 1$ is immediately to the left of $n$ then $\ell+1, n, \ell$ is a copy of $\fish$, which is forbidden.
If $\ell + 1$ is farther to the left of $n$ then let $r$ be the entry immediately to the right of $\ell + 1$.
If $r < \ell$ then $\ell+1, r, \ell$ is a copy of 312, which contradicts our assumption that $\pi$ avoids $312$.
If $r > \ell + 1$ then $\ell+1, r, \ell$ is a copy of $\fish$, which is forbidden.
Therefore, $\pi_n = n-1$ and $\pi$ satisfies (1a).
\end{proof}

In order to describe the generating tree for $F_n(321,3124)$, we need to show that if $\pi \in F_n(321, 3124)$ has label (1b) then it has exactly one active site, and inserting $n+1$ into that site produces another permutation with label (1b).
The second part will follow if we can establish the first part, since inserting $n+1$ cannot remove copies of $312$.
We establish the first part in a sequence of three lemmas.

\begin{lemma}
\label{lem:1bnnotend}
Suppose $n \ge 3$ and $\pi \in F_n(321,3124)$ has label (1b).
Then $\pi$ has at least one entry to the right of $n$.
\end{lemma}
\begin{proof}
If $\pi$ does not have an entry to the right of $n$ then $\pi_n = n$.
In this case each copy of $312$ in $\pi$ is completely to the left of $n$.
Therefore, each copy of $312$ produces a copy of $3124$ by appending $n$, which contradicts the fact that $\pi$ avoids $3124$.
\end{proof}

\begin{lemma}
\label{lem:1bntoprightmostdescent}
Suppose $n \ge 3$ and $\pi \in F_n(321, 3124)$ has label (1b).
Then $\pi$ contains a descent and the left entry of the rightmost descent in $\pi$ is $n$.
\end{lemma}
\begin{proof}
Since $\pi$ has label (1b) it contains $312$, so it cannot be $12\cdots n$.
Therefore $\pi$ contains a descent.

Now consider where the $n$ is located in $\pi$.
By Lemma \ref{lem:1bnnotend} there is at least one entry to its right, so it is the left entry of a descent.
And since $\pi$ avoids $321$, the entries to the right of $n$ must be increasing.
Therefore, $n$ is the left entry of the rightmost descent.
\end{proof}

\begin{lemma}
\label{lem:1btwopossibleactivesites}
Suppose $n \ge 3$ and $\pi \in F_n(321, 3124)$ has label (1b).
Let $x$ be the entry immediately to the right of $n$ in $\pi$, let $a$ be the site between $n$ and $x$, and let $b$ be the site immediately to the right of $x$.
Then the following hold.
\begin{enumerate}
\item[{\upshape (i)}]
$a$ is active if and only if $n-1$ is to the left of $n$.
\item[{\upshape (ii)}]
$b$ is active if and only if $n-1$ is to the right of $n$.
\item[{\upshape (iii)}]
No site other than $a$ and $b$ is active in $\pi$.
\end{enumerate}
In particular, $\pi$ has exactly one active site.
\end{lemma}
\begin{proof}
(i)
Suppose we insert $n+1$ in $a$.
If the resulting permutation contains a copy $uvw$ of $321$ then we must have $u = n+1$.
In that case, $n,v,w$ is a copy of $321$ in $\pi$, which contradicts the fact that $\pi$ avoids $321$.
Similarly, if inserting $n+1$ in $a$ creates a copy of $3124$ then $\pi$ also contains $3124$, contradicting the fact that $\pi$ does not contain $3124$.
Therefore, $a$ is active if and only if inserting $n+1$ in $a$ does not create a copy of $\fish$.
But this happens if and only if $n-1$ is to the left of $n$.

(ii)
In both directions we prove the contrapositive.

($\Rightarrow$)
Suppose $n-1$ is to the left of $n$ and set $c = \pi_n$.
Since $\pi$ avoids $321$ the entries to the right of $n$ are in increasing order, so $c+1$ is to the left of $n$.
Let $y$ be the entry immediately to the right of $c+1$.
If $y > c+1$ then $c+1, y, c$ is a copy of $\fish$ in $\pi$, which is a contradiction.
If $x < y < c$ then $c+1, y, x$ is a copy of $321$ in $\pi$, which is also a contradiction.
Therefore we must have $y < x$.
In this case $c+1, y, x$ is a copy of $312$, so inserting $n+1$ in $b$ produces a copy of $3124$, which means $b$ is not active. 

($\Leftarrow$)
Suppose $b$ is not active.
Since $x$ and the entries to the right of $b$ form an increasing sequence, inserting $n+1$ in $b$ cannot create a copy of $321$ or $\fish$, so it must create a copy $tuvw$ of $3124$.
Note that we must have $w = n+1$, since we would otherwise have a copy of $3124$ in $\pi$.
In addition, we must also have $v = x$;  if this were not the case, then $t, u, v, n$ would be a copy of $3124$ in $\pi$.
Now if $n-1$ is to the right of $n$ then $t,u,x,n-1$ is a copy of $3124$ in $\pi$, so $n-1$ must be to the left of $n$.

(iii)
If we insert $n+1$ in any site to the left of $n$, then $n+1,n,x$ will be a copy of $321$, so no site to the left of $n$ is active.
And since $\pi$ contains $312$, inserting $n+1$ in any site to the right of $b$ will create a copy of $3124$, so no site to the right of $b$ is active.
\end{proof}

We can now describe the active sites in each $\pi \in F_n(321,3124)$ according to $\pi$'s label.
More specifically, as we did in Proposition \ref{prop:F3211423rules}, we label each active site in $\pi$ with the label of the permutation we obtain by inserting $n+1$ into that site.

\begin{proposition}
\label{prop:F3213124rules}
Suppose $n \ge 1$ and $\pi \in F_n(321, 3124)$.
Then the following hold.
\begin{enumerate}
\item[{\upshape (1a)}]
If $\pi$ has label (1a) then $\pi = \cdots n\ n-1^2$.
\item[{\upshape (1b)}]
If $\pi$ has label (1b) then $\pi$ has exactly one active site and inserting $n+1$ into that site produces a permutation with label (1b).
\item[{\upshape (k)}]
If $k \ge 2$ and $\pi$ has label (k) then $\pi = \cdots a^{1b} n-k+2^{1b} n-k+3^{1b} \cdots n-2^{1b} n-1^{1a} n^{k+1}$, where $a < n-k+1$.
\end{enumerate}
\end{proposition}
\begin{proof}
(1a)
If $\pi$ has label (1a) then inserting $n+1$ anywhere to the left of $n$ will create a copy of $321$, so no site left of $n$ is active.
On the other hand, inserting $n+1$ between $n$ and $n-1$ will create a copy of $\fish$, so the site between $n$ and $n-1$ is not active.
Finally, since $\pi$ does not contain $312$, inserting $n+1$ at the right end of $\pi$ cannot create a forbidden pattern, so this site is active.
Inserting $n+1$ also cannot create a copy of $312$, so the resulting permutation has label (2).

(1b)
This follows from Lemma \ref{lem:1btwopossibleactivesites}.

(k)
To show that every site to the right of $a$ is active, first note that since $\pi$ avoids $312$, inserting $n+1$ into any site cannot create a copy of $3124$.
Since the entries from $a$ to the end of $\pi$ are in increasing order, inserting $n+1$ into any of the sites to the right of $a$ cannot create a copy of $\fish$ or $321$, so all of these sites are active.

To show that no site to the left of $a$ is active, let $\sigma \in F_{n-k+1}(321,3124)$ be the permutation we obtain from $\pi$ by removing all of the entries to the right of $a$.
Since $a < n-k+1$, we know $\sigma$ does not end with its largest entry, and since $\pi$ avoids $312$, we know $\sigma$ most also avoid $312$.
Therefore, $\sigma$ has label (1a) and its rightmost two entries are $n-k+1,n-k$.
Therefore, inserting $n+1$ in any site to the left of $a$ creates either a copy of $321$ or $\fish$, so none of these sites is active in $\pi$.
\end{proof}

In view of Proposition \ref{prop:F3213124rules}(1a) and the fact that the permutation $1$ has label (2), we take the empty permutation to have label (1a).

Taken together, Propositions \ref{prop:F3213124labels} and \ref{prop:F3213124rules} show that the generating tree for the Fishburn permutations which avoid $321$ and $3124$ has the following description.

\begin{proposition}
\label{prop:F3213124gentree}
The generating tree for the Fishburn permutations which avoid $321$ and $3124$ is given by the following.

\medskip

\begin{tabular}{lll}
Root: & & $(2)$ \\
& & \\
Rules: & & $(1a) \rightarrow (2)$  \\[2ex]
& & $(1b) \rightarrow (1b)$ \\[2ex]
& & $(k) \rightarrow \underbrace{(1b),\ldots,(1b)}_{k-2}, (1a), (k+1)$
\end{tabular}
\end{proposition}

We will use our description of the generating tree for $F_n(321,3124)$ to study generating functions for these permutations with respect to number of left-to-right maxima.
(Unfortunately, knowing the label, length, and inversion number of a permutation is not enough to determine the inversion numbers of its children, so we cannot use this approach to find generating functions with respect to inversion number.)
We begin by setting some notation.

\begin{definition}
\label{defn:3213124ogfs}
For any $n \ge 0$ and any label (y), we write $[y]_n(t)$ to denote the generating function given by
\[ [y]_n(t) = \sum_{\substack{\pi \in F_n(321, 3124) \\ \lbl(\pi) = (y)}} t^{\ltormax(\pi)} \]
and we write $[y](t,x)$ to denote the generating function given by
\[ [y](t,x) = \sum_{n=0}^\infty [y]_n(t) x^n. \]
We sometimes abbreviate $[y]_n = [y]_n(t)$ and $[y]=[y](t,x)$.
We also write $T_n(t)$ to denote the generating function given by
\[ T_n(t) = \sum_{\pi \in F_n(321,3124)} t^{\ltormax(\pi)} \]
and we write $T(t,x)$ to denote the generating function given by 
\[ T(t, x) = \sum_{n=0}^\infty T_n(t) x^n. \]
\end{definition}

Our description of the generating tree for $F_n(321, 3124)$ in Proposition \ref{prop:F3213124gentree} leads to recurrence relations for $[1a]_n$, $[1b]_n$, and $[k]_n$ for $k \ge 2$.

\begin{proposition}
Using the notation in Definition \ref{defn:3213124ogfs}, for all $n \ge 1$, we have the following.
\begin{enumerate}
\item[{\upshape (i)}]
\begin{equation}
\label{eqn:32131241a}
[1a]_n = \sum_{k=2}^\infty [k]_{n-1}.
\end{equation}
\item[{\upshape (ii)}]
\begin{equation}
\label{eqn:32131241b}
[1b]_n = t [1b]_{n-1} + \sum_{k=3}^\infty \sum_{j=1}^{k-2} t^{-j} [k]_{n-1}.
\end{equation}
\item[{\upshape (iii)}]
\begin{equation}
\label{eqn:32131242}
[2]_n = t [1a]_{n-1}.
\end{equation}
\item[{\upshape (iv)}]
For all $k \ge 3$, we have
\begin{equation}
\label{eqn:3213124k}
[k]_n = t [k-1]_{n-1}.
\end{equation}
\end{enumerate}
\end{proposition}
\begin{proof}
(i)
By Proposition \ref{prop:F3213124rules}, the permutations $\pi \in F_n(321, 3124)$ with label (1a) are exactly the permutations we obtain by inserting $n$ in the second site from the right in a permutation of length $n-1$ with label $[k]$, where $k \ge 2$.
Since each of these permutations of length $n-1$ ends with a left-to-right maximum, this insertion does not change the number of left-to-right maxima, and \eqref{eqn:32131241a} follows.

(ii)
By Proposition \ref{prop:F3213124rules}, the permutations $\pi \in F_n(321, 3124)$ with label (1b) are exactly the permutations we obtain by either inserting $n$ into the active site of a permutation of length $n-1$ with label (1b) or by inserting $n$ into one of the leftmost $k-2$ active sites in a permutation of length $n-1$ with label $[k]$, where $k \ge 3$.
In the former case we are inserting $n$ to the right of $n-1$ by Lemma \ref{lem:1btwopossibleactivesites}, so the number of left-to-right maxima increases by one.
In the latter case, suppose we number the active sites we are using $1,2,\ldots,k-2$ from right to left.
For $1 \le j \le k-2$, there are $j+1$ left-to-right maxima to the right of the active site numbered $j$.
Therefore, inserting $n$ into the active site numbered $j$ reduces the number of left-to-right maxima by $j$.
Combining these observations, we obtain \eqref{eqn:32131241b}.

(iii),(iv)
These are similar to the proof of (i).
\end{proof}

To obtain formulas for our various generating functions $[y](t,x)$, it's helpful to write $[y]_n(t)$ in terms of $[1a]_n(t)$ for various labels (y).
To do this, it's useful to extend our definition of $[1a]_n(t)$ to negative $n$.

\begin{definition}
\label{defn:1anegn}
For all $n < 0$ we define $[1a]_n(t) = 0$.
\end{definition}

\begin{proposition}
\label{prop:kintermsof1a}
Using the notation in Definitions \ref{defn:3213124ogfs} and \ref{defn:1anegn}, for all $k \ge 2$ and all $n \ge 1$ we have
\begin{equation}
\label{eqn:kintermsof1a}
[k]_n = t^{k-1} [1a]_{n-k+1}.
\end{equation}
\end{proposition}
\begin{proof}
When $k=2$ this is a restatement of \eqref{eqn:32131242}, so suppose $k \ge 3$.
Now the result follows from \eqref{eqn:3213124k} by induction on $k$.
\end{proof}

Proposition \ref{prop:kintermsof1a} enables us to find a simple recurrence relation for $[1a]_n$ and expressions for the generating functions $[1a]$ and $[k]$.

\begin{proposition}
\label{prop:1a1ak}
Using the notation in Definitions \ref{defn:3213124ogfs} and \ref{defn:1anegn}, we have the following.
\begin{enumerate}
\item[{\upshape (i)}]
For all $n \ge 2$ we have
\begin{equation}
\label{eqn:1angf}
[1a]_n = t [1a]_{n-1} + t [1a]_{n-2}.
\end{equation}
\item[{\upshape (ii)}]
\begin{equation}
\label{eqn:1axtgf}
[1a] = \frac{1-tx}{1-tx-tx^2}.
\end{equation}
\item[{\upshape (iii)}]
For all $k \ge 2$ we have
\begin{equation}
\label{eqn:kxtgf}
[k] = \frac{t^{k-1} x^{k-1}(1-tx)}{1-tx-tx^2}.
\end{equation}
\end{enumerate}
\end{proposition}
\begin{proof}
(i)
We can check the result directly when $n \le 3$, so suppose $n > 3$.
Combining \eqref{eqn:32131241a} and \eqref{eqn:kintermsof1a}, we find
\[ [1a]_n = \sum_{k=2}^\infty t^{k-1} [1a]_{n-k} \]
and
\[ [1a]_{n-1} = \sum_{k=2}^\infty t^{k-1} [1a]_{n-k-1}. \]
Therefore $[1a]_n - t [1a]_{n-1} = t [1a]_{n-2}$, and the result follows.

(ii)
When we multiply both sides of \eqref{eqn:1angf} by $x^n$, sum over all $n \ge 2$, and use the fact that $[1a]_0 = 1$ and $[1a]_1 = 0$, we find
\[ [1a] -1 = t x([1a] - 1) + t x^2 [1a]. \]
When we solve for $[1a]$ we obtain the given formula.

(iii)
This follows from (ii) and \eqref{eqn:kintermsof1a}.
\end{proof}

In the same way, we can find a formula for the generating function $[1b]$.

\begin{proposition}
\label{prop:32131241b}
Using the notation in Definition \ref{defn:3213124ogfs}, we have 
\begin{equation}
\label{eqn:1bformula}
[1b] = \frac{t x^3}{(1-tx-tx^2)(1-tx)(1-x)}.
\end{equation}
\end{proposition}
\begin{proof}
When we multiply \eqref{eqn:32131241b} by $x^n$, sum over $n \ge 1$, and use the fact that $[1b]_0 = [1b]_1 = 0$, we find
\[ [1b] = t x [1b] + \sum_{n=1}^\infty \sum_{k=3}^\infty \sum_{j=1}^{k-2} t^{-j} [k]_{n-1} x^n. \]
Reordering the sum on the right and using \eqref{eqn:kxtgf}, we find
\begin{align*}
[1b] &= t x [1b] + \sum_{k=3}^\infty \sum_{j=1}^{k-2} x t^{-j} \left( \frac{t^{k-1} x^{k-1} (1-tx)}{1-tx-tx^2}\right) \\
&= t x [1b] + \frac{x(1-tx)}{1-tx-tx^2} \sum_{k=3}^\infty (tx)^{k-1} \sum_{j=1}^{k-2} t^{-j} \\
&= t x [1b] + \frac{x(1-tx)}{1-tx-tx^2} \left( \sum_{k=3}^\infty \frac{tx^{k-1}}{1-t} + \sum_{k=3}^\infty \frac{(tx)^{k-1}}{t-1}\right) \\
&= t x [1b] +  \frac{tx^3}{(1-x)(1-tx)}.
\end{align*}
Now when we solve for $[1b]$ we obtain \eqref{eqn:1bformula}.
\end{proof}

Next we use Propositions \ref{prop:1a1ak} and \ref{prop:32131241b} to find a formula for the generating function $T(t,x)$.

\begin{proposition}
Using the notation in Definition \ref{defn:3213124ogfs}, we have 
\begin{equation}
\label{eqn:Ttx3213124}
T(t,x) = \frac{1-(t+1)x+t x^2+t x^3}{(1-tx-tx^2)(1-tx)(1-x)}.
\end{equation}
\end{proposition}
\begin{proof}
Combine the fact that
\[ T(t,x) = [1a] + [1b] + \sum_{k=2}^\infty [k] \]
with \eqref{eqn:1axtgf}, \eqref{eqn:1bformula}, and \eqref{eqn:kxtgf} and simplify the result.
\end{proof}

We can now use our generating functions to obtain $|F_n(321,3124)|$, as well as to find the number of permutations in $F_n(321,3124)$ with each label.
It will be helpful to introduce notation for these numbers.

\begin{definition}
\label{defn:yn3213124}
For each label (y) and each $n \ge 0$, we write $(y)_n$ to denote the number of permutations in $F_n(321,3124)$ with label (y).
We note that $(y)_n = [y]_n(1)$.
\end{definition}

Some of our formulas will involve Fibonacci numbers with negative indices.
With this in mind, we extend our definition of the Fibonacci numbers so that for $n < -2$ we have $F_n = 0$, $F_{-2} = 1$, $F_{-1} = 0$, and $F_n = F_{n-1} + F_{n-2}$ for $n \ge 0$.

\begin{proposition}
Using the notation in Definition \ref{defn:yn3213124}, for all $n \ge 0$ and all $k \ge 2$ we have the following.
\begin{enumerate}
\item[{\upshape (i)}]
$(1a)_n = F_{n-2}$.
\item[{\upshape (ii)}]
$(1b)_n = F_{n+1} - n - 1$.
\item[{\upshape (iii)}]
$(k)_n = F_{n-k-1}$.
\item[{\upshape (iv)}]
$|F_n(321, 3124)| = F_{n+2}-n-1$.
\end{enumerate}
\end{proposition}
\begin{proof}
Set $t = 1$ in \eqref{eqn:1axtgf}, \eqref{eqn:1bformula}, \eqref{eqn:kxtgf}, and \eqref{eqn:Ttx3213124} and find the coefficient of $x^n$.
\end{proof}

Our formula for $|F_n(321,3124)|$ leads to several open problems, the first two of which are analogues of Open Problems \ref{op:1423binseq} and \ref{op:1423Motzkin}.

\begin{openproblem}
\label{op:3124bs}
For each $n \ge 1$ find a constructive bijection between $F_n(321, 3124)$ and the set of binary sequences of length $n -1$ in which no two consecutive $1$s have more than one $0$ between them.
\end{openproblem}

\begin{openproblem}
\label{op:3124Motzkin1}
For each $n \ge 1$, let $D_n$ be the set $F_n(321,3124)$ with the permutation $12 \cdots n$ removed.
Find a constructive bijection between $D_n$ and the set of Motzkin paths of length $n$ with exactly one ascent.
\end{openproblem}

\begin{openproblem}
\label{op:3124Motzkin2}
For each $n \ge 1$, find a bijection between the set of permutations in $F_{n+1}(321,3124)$ with label (1b) and the set of Motzkin paths of length $n$ with exactly one ascent.
\end{openproblem}

\section{Enumerating $F_n(321,2143)$}
\label{sec:Fn3212143}

Instead of the Fibonacci numbers, our last enumerative results involve binomial coefficients and powers of two.
In particular, we show that $|F_n(321,2143)| = 2^{n-1}$ for all $n \ge 1$.

The above claim would follow from \cite[Proposition 7]{SS} if we had $F_n(321,2143) = S_n(321,231)$, but there are several ways to see we don't, including the fact that $2143 \in S_4(321,231)$ but $2143 \not\in F_4(321,2143)$.
And the problem only gets worse as $n$ grows.
In particular, because $|F_n(321,2143) \cap S_n(321,231)| = |S_n(321,231,2143)| = \binom{n}{2} + 1$, which grows like $n^2$, as $n$ grows we find most permutations in $F_n(321,2143)$ are not in $S_n(321,231)$.

We actually prove $|F_n(321,2143)| = 2^{n-1}$ by proving a refinement, which says the number of $\pi \in F_n(321,2143)$ with exactly $k$ left-to-right maxima is $\binom{n-1}{k-1}$.
To do this, we start with the generating tree for the Fishburn permutations which avoid $321$ and $2143$.

\begin{proposition}
\label{prop:3212143labels}
Suppose $n \ge 1$ and $\pi \in F_n(321,2143)$.
Then exactly one of the following holds.
\begin{enumerate}
\item[{\upshape (1a)}]
$\pi_n = n$ and $\pi$ has a descent.
\item[{\upshape (k)}]
$k \ge 2$, $n = k-1$, and $\pi$ has no descent.
\item[{\upshape (k*)}]
The rightmost $k+1$ entries of $\pi$ are $n,n-k,n-k+1,\ldots,n-2,n-1$ or the rightmost $k$ entries of $\pi$ are $n,\ell-k+2, \ell-k+3,\ldots,\ell-1,\ell$ for some $\ell$ with $k-1 < \ell < n-1$.
\end{enumerate}
\end{proposition}
\begin{proof}
We note that the given cases are mutually exclusive, so it's sufficient to show each $\pi \in F_n(321,2143)$ belongs to one of them.

If $\pi$ has no descents then $\pi = 12\cdots n$ and $\pi$ is in case (k).

If $\pi$ has a descent and $\pi_n = n$ then $\pi$ is in case (1a), so suppose $\pi$ has a descent but $\pi_n \neq n$.
Since $\pi$ avoids $321$, the entries to the right of $n$ must be in increasing order.
Suppose these entries are $a_1 < a_2 < \cdots < a_m$.
If $a_2 > a_1 + 1$ then $a_1 + 1$ is to the left of $n$ in $\pi$.
Let $x$ be the entry immediately to the right of $a_1 + 1$.
If $x > a_1 + 1$ then $a_1 + 1, x, a$ is a copy of $\fish$.
On the other hand, if $x < a_1 + 1$ then $a_1 + 1, x, n, a_2$ is a copy of $2143$.
Since $\pi$ avoids both $\fish$ and $2143$, we must have $a_2 = a_1 + 1$.
Similarly, $a_{j+1} = a_j + 1$ for $1 \le j \le m-1$, so $\pi$ is in case (k*) for some $k$.
\end{proof}

We can now describe the active sites in each $\pi \in F_n(321,2143)$ according to $\pi$'s label.
More specifically, as we did in Propositions \ref{prop:F3211423rules} and \ref{prop:F3213124rules}, we label each active site in $\pi$ with the label of the permutation we obtain by inserting $n+1$ into that site.

\begin{proposition}
\label{prop:3212143rules}
Suppose $n \ge 1$ and $\pi \in F_n(321,2143)$.
Then the following hold.
\begin{enumerate}
\item[{\upshape (1a)}]
If $\pi$ has label (1a) then $\pi = \cdots n^{1a}$.
\item[{\upshape (k)}]
If $\pi$ has label (k) for $k \ge 2$ then $\pi = ^{k-1*}1^{k-2*} 2^{k-3*} \cdots ^{2*}k-2^{1*} k-1^{k+1}$.
\item[{\upshape (k*)}]
If $\pi$ has label (k*) for $k \ge 1$ then $\pi = \cdots n\ n-k^{k*} n-k+1^{k-1*} \cdots ^{3*}n-2^{2*}n-1^{1a}$ or $\pi = \cdots n^{k*} \ell-k+2^{k-1*} \ell-k+3^{k-2*} \cdots ^{3*}\ell-1^{2*}\ell^{1a}$.
\end{enumerate}
\end{proposition}
\begin{proof}
(1a)
Suppose $\pi_n = n$ and $\pi$ has a descent.
We first note that the rightmost site in $\pi$ is active, since none of our forbidden patterns ends with its largest entry, and inserting $n+1$ into this site produces a permutation with label (1a).

To show no other sites in $\pi$ are active, first note that no site to the right of a descent can be active if there is an entry to its right, since insertion into the site will use $n$ to create a copy of $2143$.
Similarly, no site to the left of a descent can be active since insertion into the site will create a copy of $321$.

Now suppose $\pi$ has a unique descent, let $a$ be the left entry in the descent, and let $b$ be the largest number to the right of $a$ which is less than $a$.
We know $b$ exists because the number immediately to the right of $a$ is less than $a$.
Now either $b+1 = a$ or $b+1$ is to the left of $a$, by our choice of $b$.
If $b+1 = a$ then inserting $n+1$ in the descent creates the subsequence $b+1, n+1, b$, which is a copy of $\fish$, so the site in the descent is not active.

Now suppose $b+1$ is to the left of $a$ and $c$ is the entry immediately to the right of $b+1$.
If $c > b+1$ then $b+1,c,b$ is a copy of $\fish$, which contradicts the fact that $\pi$ avoids $\fish$.
If $c < b+1$ then inserting $n+1$ in the descent creates the subsequence $b+1, c, n+1, n$, which is a copy of $2143$, so the site in the descent is not active.

(k)
We first note that inserting $n+1$ in $\pi$ cannot create a forbidden pattern, since removing the largest entry from any of our forbidden subsequences leaves a permutation with a descent, and $\pi$ has no descents.
Therefore, every site in $\pi$ is active.
Now we can check that inserting $n+1$ produces permutations with the given labels.

(k*)
We first note that inserting $n+1$ in any site to the left of $n$ will create a copy of $321$, so no site to the left of $n$ can be active.

In the first case the site immediately to the right of $n$ is also not active because inserting $n+1$ into this site will create a copy of $\fish$.
In the second case $n-1$ is to the right of $n$, so inserting $n+1$ immediately to the right of $n$ does not create a copy of $\fish$.
Since the entries to the right of $n$ are increasing, inserting $n+1$ immediately to the right of $n$ also does not create a copy of $321$.
And if inserting $n+1$ in this site created a copy of $2143$, then we could replace $n+1$ with $n$ in that copy to obtain a copy of $2143$ in $\pi$, which contradicts the fact that $\pi$ avoids $2143$.
Therefore, in the second case the site immediately to the right of $n$ is active.

Now consider the sites to the right of $n$ which are not adjacent to $n$.
Since the entries to the right of $n$ are increasing, inserting $n+1$ in one of them cannot create a copy of $\fish$ or $321$.
If inserting $n+1$ into one of these sites created a copy $a,b,n+1,d$ of $2143$, then $a$ and $b$ would be to the left of $n$, since the entries to the right of $n$ are consecutive integers.
In that case, $a,b,n,d$ would be a copy of $2143$ in $\pi$, which contradicts the fact that $\pi$ avoids $2143$.
It follows that all of the sites to the right of $n$ which are not adjacent to $n$ are active.

Now we can check that inserting $n+1$ produces permutations with the given labels.
\end{proof}

Taken together, Propositions \ref{prop:3212143labels} and \ref{prop:3212143rules} show that the generating tree for the Fishburn permutations which avoid $321$ and $2143$ has the following description.

\begin{proposition}
\label{prop:F3212143gentree}
The generating tree for the Fishburn permutations which avoid $321$ and $2143$ is given by the following.

\medskip

\begin{tabular}{lll}
Root: & & $(2)$ \\
& & \\
Rules: & & $(1a) \rightarrow (1a)$  \\[2ex]
& & $(k) \rightarrow (1*), (2*), \ldots, (k-1*), (k+1)$ \\[2ex]
& & $(k*) \rightarrow (1a), (2*), (3*), \ldots, (k*)$
\end{tabular}
\end{proposition}

The forms of the rules in Proposition \ref{prop:3212143rules} enable us to give recurrence relations for the generating functions with respect to both inversion number and number of left-to-right maxima.
To give these recurrence relations, we first need some notation.

\begin{definition}
\label{defn:3212143notation}
For any $n \ge 0$ and any label (y), we write $[y]_n(q,t)$ to denote the generating function given by
\[ [y]_n(q,t) = \sum_{\pi \in F_n(321,2143)} q^{\inv(\pi)} t^{\ltormax(\pi)}. \]
We sometimes abbreviate $[y]_n = [y]_n(q,t)$.
We also write $T_n(q,t)$ to denote the generating function given by
\[ T_n(q,t) = \sum_{\pi \in F_n(321,2143)} q^{\inv(\pi)} t^{\ltormax(\pi)}. \]
\end{definition}

In Table \ref{table:F3212143data} we have $[y]_n$ for various $y$ and $2 \le n \le 5$.

\begin{table}[ht]
\centering
\begin{tabular}{c|ccccc}
$n$ & $[1a]_n$ & $[3]_n$ & $[4]_n$ & $[5]_n$ & $[6]_n$ \\
\hline
$2$ & $qt$ & $t^2$ & 0 & 0 & 0 \\
$3$ & $qt^2$ & 0 & $t^3$ & 0 & 0 \\
$4$ & $2qt^3 + q^2 t^2$ & 0 & 0 & $t^4$ & 0 \\
$5$ & $3qt^4 + 2q^2t^3+q^3t^3+q^3t^2$ & 0 & 0 & 0 & $t^5$ \\
\end{tabular}

\bigskip

\begin{tabular}{c|cccc}
$n$ & $[1*]_n$ & $[2*]_n$ & $[3*]_n$ & $[4*]_n$ \\
\hline
$2$ & 0 & 0 & 0 & 0 \\
$3$ & $qt^2$ & $q^2 t$ & 0 & 0  \\
$4$ & $qt^3$ & $q^2 t^2 + q^3 t^2$ & $q^3 t$ & 0 \\
$5$ & $qt^4$ & $q^2t^3 + q^3t^3+q^4t^3+q^4t^2$ & $q^3 t^2 + q^5t^2$ & $q^4 t$ \\
\end{tabular}
\caption{The polynomials $[y]_n(q,t)$ for $2 \le n \le 5$.}
\label{table:F3212143data}
\end{table}

As promised, we can use Proposition \ref{prop:3212143rules} to give recurrence relations for $[y]_n$.

\begin{proposition}
\label{prop:3212143recurrences}
Using the notation in Definition \ref{defn:3212143notation}, we have the following.
\begin{enumerate}
\item[{\upshape (i)}]
For all $n \ge 3$ we have
\begin{equation}
\label{eqn:1a}
[1a]_n = t [1a]_{n-1} + \sum_{j=1}^\infty t [j*]_{n-1}.
\end{equation}
\item[{\upshape (ii)}]
For all $k \ge 3$ and all $n \ge 2$ we have
\begin{equation}
\label{eqn:kn}
[k]_n = \begin{cases} t^n & k = n+1; \\ 0 & k \neq n+1. \end{cases}
\end{equation}
\item[{\upshape (iii)}]
For all $n \ge 3$ we have
\begin{equation}
\label{eqn:1star}
[1*]_n = qt^{n-1}.
\end{equation}
\item[{\upshape (iv)}]
For all $n \ge 3$ and all $k$ with $1 \le k \le n$ we have
\begin{equation}
\label{eqn:kstarn}
[k*]_n = q^k t^{n-k} + \sum_{j=k}^\infty q^{k-1} t [j*]_{n-1}.
\end{equation}
\end{enumerate}
\end{proposition}
\begin{proof}
(i)
Each permutation in $F_{n-1}(321,2143)$ with label (1a) produces exactly one permutation in $F_n(321,2143)$ with label (1a), which has one additional left-to-right maximum and no additional inversions.
Similarly, each permutation in $F_{n-1}(321,2143)$ with label (j*) for some $j \ge 1$ also produces exactly one permutation in $F_n(321, 2143)$ with label (1a), which has one additional left-to-right maximum and no additional inversions.
Since these are the only permutations in $F_{n-1}(321, 2143)$ which produce permutations on $F_n(321,2143)$ with label (1a), line \eqref{eqn:1a} follows.

(ii)
For each $k \ge 3$ the only permutation with label $k$ is $1 2 \cdots k-2\ k-1$, so line \eqref{eqn:kn} follows.

(iii)
The only permutation in $F_{n-1}(321, 2143)$ which produces a permutation in $F_n(321,2143)$ with label (1*) is $1 2 \cdots n-2\ n-1$.
This permutation produces exactly one permutation with label (1*), which has $n-1$ left-to-right maxima and one inversion.
Now line \eqref{eqn:1star} follows.

(iv)
This is similar to the proof of (i), using \eqref{eqn:kn}.
\end{proof}

We have not been able to find simple expressions for $[k*]_n$, or even for the generating function for $[k*]_n$.
However, we can find a simple expression for $[y]_n(1,t)$ for any label (y), which we can use to find a simple expression for $T_n(1,t)$.

\begin{proposition}
\label{prop:3212143t1ogfs}
Using the notation in Definition \ref{defn:3212143notation}, we have the following.
\begin{enumerate}
\item[{\upshape (i)}]
For all $n \ge 3$ we have
\begin{equation}
\label{eqn:1aformula}
[1a]_n(1,t) = t^2(t+1)^{n-2} - t^n.
\end{equation}
\item[{\upshape (ii)}]
For all $n \ge 3$ and all $k \ge 2$ we have
\begin{equation}
\label{eqn:kstarformula}
[k*]_n(1,t) = \begin{cases} t^2(t+1)^{n-k-1} - t^{n-k+1} + t^{n-k} & \text{if\ $k < n$}; \\ 0 & \text{if\ $k \ge n$}.\end{cases}
\end{equation}
\item[{\upshape (iii)}]
For all $n \ge 1$ the generating function for $F_n(321, 2143)$ with respect to left-to-right maxima is
\begin{equation}
\label{eqn:3212143gf}
T_n(1,t) = t (t + 1)^{n-1}.
\end{equation}
\end{enumerate}
\end{proposition}
\begin{proof}
We can check (i) and (ii) directly for $n \le 5$, so suppose $n \ge 6$;  we prove (i) and (ii) together by induction on $n$.

(i)
Setting $q = 1$ in \eqref{eqn:1a} and using the inductive hypothesis for (i) and (ii) we find
\begin{align*}
[1a]_n(1,t) &= t^3 (t+1)^{n-3} - t^n + \sum_{j=1}^{n-2} \left( t^3 (t+1)^{n-j-2} - t^{n-j+1} + t^{n-j} \right) \\
&= t^3 \left( \frac{(t+1)^{n-2}-1}{(t+1)-1}\right) - t^n + t^2 \\
&= t^2(t+1)^{n-2} - t^n.
\end{align*}

(ii)
The fact that $[k*]_n(1,t) = 0$ if $k \ge n$ follows from the definition of (k*) in Proposition \ref{prop:3212143labels}, so suppose $k < n$.
Setting $q = 1$ in \eqref{eqn:kstarn} and using the inductive hypothesis we find
\begin{align*}
[k*]_n(1,t) &= t^{n-k} + \sum_{j=k}^{n-2} \left( t^3 (t+1)^{n-j-2} - t^{n-j+1} + t^{n-j}\right) \\
&= t^{n-k} + t^3 \left( \frac{(t+1)^{n-k-1}-1}{(t+1)-1}\right) - t^{n-k+1} + t^2 \\
&= t^2 (t+1)^{n-k-1} - t^{n-k+1} + t^{n-k}.
\end{align*}

(iii)
We can check (iii) for $1 \le n \le 5$ using the data above, so suppose $n \ge 6$.
Using Proposition \ref{prop:3212143labels} and lines \eqref{eqn:1aformula}, \eqref{eqn:kn}, \eqref{eqn:1star}, and \eqref{eqn:kstarformula} we find
\begin{align*}
T_n(1,t) &= [1a]_n(1,t) + \sum_{j=3}^\infty [j]_n(1,t) + \sum_{j=1}^\infty [j*]_n(1,t) \\
&= t^2 (t+1)^{n-2} - t^n + t^n + t^{n-1} + \sum_{j=2}^{n-1} \left( t^2 (t+1)^{n-j-1} - t^{n-j+1} + t^{n-j}\right) \\
&= t^2(t+1)^{n-2} + t^{n-1} + t^2\left(\frac{(t+1)^{n-2}-1}{(t+1)-1}\right) - t^{n-1} + t \\
&= t(t+1)^{n-1}.
\end{align*}
\end{proof}

Proposition \ref{prop:3212143t1ogfs} has two nice enumerative corollaries.

\begin{corollary}
For all $n \ge 1$ and all $k \ge 1$, the number of permutations in $F_n(321,2143)$ with exactly $k$ left-to-right maxima is $\binom{n-1}{k-1}$.
\end{corollary}
\begin{proof}
This follows from \eqref{eqn:3212143gf} and the binomial theorem.
\end{proof}

\begin{corollary}
For all $n \ge 1$ we have
\[ |F_n(321,2143)| = 2^{n-1}. \]
\end{corollary}
\begin{proof}
Set $t = 1$ in \eqref{eqn:3212143gf}.
\end{proof}

\section{Open Problems and Conjectures}
\label{sec:opc}

The study of pattern-avoiding Fishburn permutations, pattern-avoiding ascent sequences, and the connections between them still includes numerous open problems, some of which we review in this section.
We focus on enumerative and bijective questions that do not involve statistics on permutations;  studying statistics on Fishburn permutations offers a variety of additional avenues for future work.

In addition to Fishburn permutations, pattern-avoiding ascent sequences seem to have connections with other families of permutations.
For example, several authors, beginning with Knuth \cite{Knuth}, have studied permutations which are sortable using one or more stacks, arranged in various configurations.
Building on work of Atkinson, Murphy, and Ru\v{s}kuc \cite{AMR} and Smith \cite{Smith}, Cerbai, Claesson, and Ferrari \cite{CCF} have generalized this idea.
Given a permutation $\sigma$, we can construct a sorting machine using two stacks connected in series.
We require that the entries in the first stack always avoid $\sigma$, when read from top to bottom, while the entries in the second stack must always be increasing (in order to obtain the identity permutation).
If $\pi$ is sortable using this series of stacks, then we say $\pi$ is {\em $\sigma$-machine sortable}.
Cerbai, Claesson, and Ferrari have generated data \cite[Section 6]{CCF} that support the following conjecture.

\begin{conjecture}[Cerbai, Claesson, Ferrari]
For all $n \ge 0$, the set $A_n(201)$ is in bijection with permutations in $S_n$ which are 312-machine sortable.
\end{conjecture}

By comparing the set of $312$-machine sortable permutations in $S_n$ with $A_n(201)$ for small $n$, one can confirm that neither $\fishmap^{-1}$ nor the composition of $\fishmap^{-1}$ with any of the dihedral symmetries on $S_n$ restricts to a bijection between these two sets.

In addition to the potential connections between pattern-avoiding Fishburn permutations and certain types of binary sequences and Motzkin paths noted in Open Problems \ref{op:1423binseq}, \ref{op:1423Motzkin}, \ref{op:3124bs}, \ref{op:3124Motzkin1}, and \ref{op:3124Motzkin2}, several equivalences among enumerations of pattern-avoiding Fishburn permutations conjectured by Gil and Weiner \cite{GW} remain open.
These include the following, which we have verified for $n \le 15$.

\begin{conjecture}[Gil,Weiner]
\label{conj:GW1}
For all $n \ge 0$ we have
\[ |F_n(2413)| = |F_n(2431)| = |F_n(3241)|.\]
\end{conjecture}

\begin{conjecture}[Gil,Weiner]
\label{conj:GW2}
For all $n \ge 0$ we have
\[ |F_n(3214)| = |F_n(4132)| = |F_n(4213)|. \]
\end{conjecture}

Gil and Weiner have also shown \cite{GW} that for a variety of $\sigma \in S_4$ we have $|F_n(\sigma)| = C_n$, where $C_n = \frac{1}{n+1}\binom{2n}{n}$ is the $n$th Catalan number.
Several of the sets of Fishburn permutations which avoid two or more of these patterns also appear to be counted by nice sequences.
Some of these are captured in the following conjectures, which we have verified for $n \le 17$.

\begin{conjecture}
\label{conj:threepairs}
For all $n \ge 1$,
\[ |F_n(1324, 2143)| = |F_n(1423,2143)| = |F_n(1423,3124)| = (n-1) 2^{n-2}+1.\]
\end{conjecture}

\begin{conjecture}
\label{conj:otherpairs}
For all $n \ge 1$,
\[ |F_n(1324, 1423)| = |F_n(1324,3124)| = F_{2n-2}. \]
\end{conjecture}

\begin{conjecture}
\label{conj:binomn13}
For all $n \ge 1$,
\[ |F_n(2143, 1423, 3124)| = 2 \binom{n+1}{3} + n + 1. \]
\end{conjecture}

\begin{conjecture}
\label{conj:Grassmann}
For all $n \ge 1$,
\[ |F_n(1324,2143,1423)| = |F_n(1324,2143,3124)| = |F_n(1324,1423,3124)| = 2^n - n. \]
\end{conjecture}

In connection with Conjecture \ref{conj:Grassmann}, we note that $2^n-n$ is also the number of Grassmannian permutations (that is, permutations with at most one descent) of length $n$.

\begin{conjecture}
For all $n \ge 1$,
\[ |F_n(1324, 2143, 1423, 3124)| = \frac{(n+2)(n^2-2n+3)}{6}.\]
\end{conjecture}

Last but not least in this family of conjectures, we have the following.

\begin{conjecture}
\label{conj:FS}
For all $n \ge 1$,
\[ |F_n(2143,3124)| = |S_n(231,4123)|. \]
\end{conjecture}

We have seen in Section \ref{sec:Fishisirrelevant} that in certain cases $F_n(\sigma_1,\ldots,\sigma_k)$ is actually $S_n(\tau_1,\ldots,\tau_\ell)$ for some $\tau_1,\ldots,\tau_\ell$.
In the cases we saw in that section, $\tau_1,\ldots,\tau_\ell$ were $231,\sigma_1,\ldots,\sigma_k$.
Conjecture \ref{conj:FS} suggests there are other situations in which $|F_n(\sigma_1,\ldots,\sigma_k)| = |S_n(\tau_1,\ldots, \tau_\ell)|$.
We have several other conjectured coincidences of this sort, though in these cases we generally do not have $F_n(\sigma_1,\ldots,\sigma_k) = S_n(\tau_1,\ldots, \tau_\ell)$.
The first of these is closely related to Conjecture \ref{conj:binomn13}.

\begin{conjecture}
\label{conj:FSSn}
For all $n \ge 0$,
\[ |F_n(2143, 1423, 3124)| = |S_n(321, 2143, 3124)| = |S_n(231, 4132, 2134)|.\]
\end{conjecture}

We have verified Conjecture \ref{conj:FSSn} for $n \le 17$.

\begin{conjecture}
\label{conj:FSn}
For all $n \ge 0$,
\[ |F_n(1243, 2134)| = |S_n(123, 3241)|. \]
\end{conjecture}

We have verified Conjecture \ref{conj:FSn} for $n \le 15$.

Baxter and Pudwell \cite[Proposition 12]{BP} have shown that $S_n(123, 3241)$ is in bijection with the set of ascent sequences of length $n$ which avoid $021$ and $102$ by showing that the number of elements in each of these sets is $3\cdot 2^{n-1} -\binom{n+1}{2} -1$, but their methods do not lead to a bijection between the two sets.
This sequence is sequence A116702 in the OEIS.

\begin{conjecture}
\label{conj:FSn2}
For all $n \ge 0$,
\[ |F_n(1243, 3124)| = |S_n(231,4123)|. \]
\end{conjecture}

We have verified Conjecture \ref{conj:FSn2} for $n \le 15$.

Baxter and Pudwell \cite[Proposition 14]{BP} have shown that $S_n(231,4123)$ is in bijection with the set of ascent sequences of length $n$ which avoid $101$ and $120$ by showing these two sets have isomorphic generating trees.
The sequence of sizes of these sets is sequence A116703 in the OEIS.

Gil and Weiner have also shown \cite[Theorem 3.3]{GW} that $|F_n(1342)|$ is a binomial transform of the sequence of Catalan numbers.
That is,
\[ |F_n(1342)| = \sum_{k=1}^n \binom{n-1}{k-1} C_{n-k}. \]
This sequence, which is sequence A007317 in the OEIS, has a variety of other combinatorial interpretations.
Some of the objects involved are Schr\"oder paths, binary trees, Euler trees, symmetric hex trees, composihedra, and Motzkin paths.
The sequence also appears in our next conjecture.

\begin{conjecture}
\label{conj:CatConv}
For all $n \ge 1$,
\[ |F_n(2413, 2431)| = |F_n(2431, 3241)| = \sum_{k=1}^n \binom{n-1}{k-1} C_{n-k}. \]
Here $C_n = \frac{1}{n+1} \binom{2n}{n}$ is the $n$th Catalan number.
\end{conjecture}

We have verified Conjecture \ref{conj:CatConv} for $n \le 14$.

Other sets of pattern-avoiding Fishburn permutations appear to be related to certain sets of pattern-avoiding involutions.
For example, the current author has shown \cite{EggeInv} that if $\sigma$ is any of the 22 permutations in Table \ref{table:Pellperms}, then the number of involutions of length $n$ which avoid $3412$ and $\sigma$ is $\frac{P_n + P_{n-1}+1}{2}$.
Here $P_n$ is the $n$th Pell number:  $P_0 = 0$, $P_1 = 1$, and $P_n = 2 P_{n-1} + P_{n-2}$ for $n \ge 2$.
\begin{table}[ht]
\centering
\begin{tabular}{|c|c|c|c|c|c|c|c|c|c|c|}
\hline
3421 & 4312 & 32541 & 52143 & 51432 & 43251 & 25431 & 53214 & 14352 & 15324 & 41325 \\
\hline
24315 & 21534 & 23154 & 21453 & 31254 & 13542 & 15243 & 42135 & 32415 & 54231 & 53421 \\
\hline
\end{tabular}
\caption{Forbidden patterns related to Conjecture \ref{conj:Pell}.}
\label{table:Pellperms}
\end{table}
This is sequence A024537 in the OEIS, which also appears in our next conjecture.

\begin{conjecture}
\label{conj:Pell}
For all $n \ge 1$, 
\[ |F_n(321,31452)| = |F_n(321,31524)| = |F_n(321,41523)| = \frac{P_n + P_{n-1}+1}{2}. \]
In particular, for each permutation $\sigma$ in Table \ref{table:Pellperms}, the number of involutions of length $n$ which avoid $3412$ and $\sigma$ is equal to $|F_n(321,31452)|$, $|F_n(321,31524)|$, and $|F_n(321,41523)|$.
\end{conjecture}

We have verified Conjecture \ref{conj:Pell} for $n \le 16$.
We note that the sequence in Conjecture \ref{conj:Pell} appears in sequence A171842 in the OEIS, which enumerates interval orders with no 3-element antichain.
Since interval orders are in bijection with Fishburn permutations, it would especially nice to find a bijection between these sets which restricts to a bijection between one of the sets in Conjecture \ref{conj:Pell} and the set of interval orders with no 3-element antichain.

For some sets $\sigma_1,\ldots,\sigma_k$ of forbidden patterns, all of the permutations in $F_n(\sigma_1,\ldots,\sigma_k)$ can be constructed in a simple way from a specific set of permutations.
In that case, finding the number of permutations of each length in this set is equivalent to enumerating $F_n(\sigma_1,\ldots,\sigma_k)$, and in several cases the numbers of these permutations are of interest in their own right.
To make these ideas precise, we first need to introduce a method of combining permutations and some notation.

For any permutations $\sigma$ and $\pi$, we write $\sigma \oplus \pi$ to denote the permutation we obtain by adding $|\sigma|$ to each entry of $\pi$ and concatenating the resulting sequences.
For example, if $\sigma = 42531$ and $\pi = 312$ then $\sigma \oplus \pi = 42531867$.

If a nonempty permutation $\pi$ cannot be written in the form $\sigma_1 \oplus \sigma_2$, where $\sigma_1$ and $\sigma_2$ are nonempty, then we say $\pi$ is {\em indecomposable}.
For each $n \ge 0$ and any set $\sigma_1, \ldots, \sigma_k$ of permutations, we write $F_n^{ind}(\sigma_1,\ldots,\sigma_k)$ to denote the set of indecomposable permutations in $F_n(\sigma_1,\ldots,\sigma_k)$.
Pattern-avoiding indecomposable permutations have been studied by Gao, Kitaev, and Zhang \cite{GKZ}, who essentially proved the following result relating $|F_n(\sigma_1,\ldots,\sigma_k)|$ and $|F_n^{ind}(\sigma_1,\ldots,\sigma_k)|$.

\begin{lemma}
(\cite[Lemma 3.1]{GKZ})
\label{lem:indecgf}
Suppose $\sigma_1,\ldots,\sigma_k$ are indecomposable and $F(x)$ and $F^{ind}(x)$ are the generating functions
\[ F(x) = \sum_{n=0}^\infty |F_n(\sigma_1,\ldots,\sigma_k)| x^n \]
and
\[ F^{ind}(x) = \sum_{n=0}^\infty |F_n^{ind}(\sigma_1,\ldots,\sigma_k)| x^n. \]
Then we have
\[ F(x) = \frac{1}{1-F^{ind}(x)} \]
and
\[ F^{ind}(x) = 1 - \frac{1}{F(x)}. \]
\end{lemma}

Gil and Weiner have noted \cite[Section 4]{GW} that $F_n(2413)$ appears to be equinumerous with $S_{n-1}(2413, 3412)$;  the corresponding sequence is sequence A165546 in the OEIS.
We have verified this conjecture for $n \le 14$.
We have an additional conjecture along these lines, which we have verified for $n \le 15$.

\begin{conjecture}
\label{conj:Fine}
For all $n \ge 1$,
\[ |F_n^{ind}(2413, 2431)| = |F_n^{ind}(2431, 3241)| = |S_{n-1}(2413, 3412, 2143)|. \]
\end{conjecture}

The sequence in Conjecture \ref{conj:Fine} is sequence A033321 in the OEIS.
It is the binomial transform of Fine's sequence, and it has a variety of combinatorial interpretations.
We note that in view of Lemma \ref{lem:indecgf}, Conjectures \ref{conj:CatConv} and \ref{conj:Fine} are equivalent.
We also note that the common refinement $|F_n^{ind}(2413, 2431, 3241)|$ appears to be given by sequence A078482 in the OEIS, which has an interpretation in terms of other pattern-avoiding permutations.

Finally, in addition the conjectures above, there are several other families of the form $F_n(321,\sigma)$ which we should be able to enumerate.
For example, we can use generating trees as we did in Sections \ref{sec:Fn3211432} and \ref{sec:Fn3213124} to show that the generating function for Fishburn permutations which avoid $321$ and $4123$ is
\[ \sum_{n=0}^\infty |F_n(321,4123)| x^n = \frac{1-x-x^2}{(1-x)(1-x-2x^2-x^3)}. \]
Therefore this sequence is sequence A262735 in the OEIS, and its terms satisfy $a_n = a_{n-1} +2 a_{n-2} + a_{n-3}-1$.
We close with the following additional conjectures, all of which we have verified for $n \le 20$.

\begin{conjecture}
We have
\begin{equation}
|F_n(321,1243)| = |F_n(321,2134)| = n^2-3n+4, \hspace{10pt} (n \ge 2);
\end{equation}
\begin{equation}
|F_n(321,1324)| = \frac32 n^2 - \frac{13}{2} n + 10, \hspace{10pt} (n \ge 3);
\end{equation}
\begin{equation}
\label{eqn:Grassmann2}
\begin{aligned}
|F_n(321,3142,2143)| = |F_n(321,1423,&\ 2143)| = |F_n(321, 2143, 3124)| \\
& = |F_n(321,2143,4123)| = \binom{n}{2} + 1, \hspace{10pt} (n \ge 0);
\end{aligned}
\end{equation}
\begin{equation}
|F_n(321,1423, 3124)| = F_n+2, \hspace{10pt} (n \ge 4);
\end{equation}
\begin{equation}
|F_n(321,1423,4123)| = |F_n(321,3124,4123)| = F_{n+1}-1, \hspace{10pt} (n \ge 1);
\end{equation}
\begin{equation}
|F_n(321,14253)| = |F_n(321,21354)| = 2^n - \binom{n}{2} -1, \hspace{10pt} (n \ge 1).
\end{equation}
\end{conjecture}

In connection with \eqref{eqn:Grassmann2}, we note that $\binom{n}{2} + 1$ is the number of Grassmannian permutations (that is, permutations with at most one descent) of length $n$ which avoid $231$.

\bigskip

\noindent
{\Large\bf Acknowledgements}

\medskip

The author thanks Giulio Cerbai, Juan Gil, Lara Pudwell, Bruce Sagan, Justin Troyka, and Yan Zhuang for valuable feedback on earlier versions of this paper.

\bibliographystyle{plain}
\bibliography{references}

\end{document}